\documentclass[11pt,a4paper,reqno,dvipsnames]{amsart}
\allowdisplaybreaks[4]

\usepackage{amsfonts}
\usepackage{amsmath}
\usepackage{amssymb}
\usepackage{amsthm}
\usepackage{amsxtra}
\usepackage{appendix}
\usepackage{bbm}
\usepackage{braket}
\usepackage{comment}
\usepackage{extarrows}
\usepackage{graphicx}
\usepackage{latexsym}
\usepackage{mathrsfs}
\usepackage{yhmath}
\usepackage{mathtools}
\let\underbrace\LaTeXunderbrace

\usepackage{float}
\usepackage{booktabs}
\usepackage{verbatim}
\usepackage{xcolor}
\usepackage[normalem]{ulem}
\usepackage{caption,subcaption}
\usepackage{hyperref} 
\hypersetup{breaklinks, raiselinks}
\hypersetup{colorlinks, citecolor=blue, linkcolor=blue, urlcolor=blue}

\usepackage{bm}

\usepackage{tikz-cd}

\usepackage[initials,lite]{amsrefs} 
\AtBeginDocument{%
    \def\MR#1{}
}

\usepackage{cleveref}

\usepackage{fullpage}

\theoremstyle{plain}
\newtheorem{Theorem}{Theorem}[section]
\newtheorem{Lemma}[Theorem]{Lemma}
\newtheorem{Corollary}[Theorem]{Corollary}
\newtheorem{Proposition}[Theorem]{Proposition}

\theoremstyle{definition}

\newtheorem{Assumptions and Discussion}[Theorem]{Assumptions and Discussion}

\newtheorem{Example}[Theorem]{Example}
\newtheorem{Definition}[Theorem]{Definition}

\newtheorem{Question}[Theorem]{Question}

\newtheorem{Remark}[Theorem]{Remark}

\newtheorem{Observation}[Theorem]{Observation}

\theoremstyle{remark}

\newtheorem{Setting}[Theorem]{Setting}

\newtheorem*{acknowledgment*}{Acknowledgment}

\usepackage[shortlabels]{enumitem}
\SetEnumerateShortLabel{a}{\textup{(\alph*)}}
\SetEnumerateShortLabel{A}{\textup{(\Alph*)}}
\SetEnumerateShortLabel{1}{\textup{(\arabic*)}}
\SetEnumerateShortLabel{i}{\textup{(\roman*)}}
\SetEnumerateShortLabel{I}{\textup{(\Roman*)}}

\def\ceil#1{\left\lceil #1 \right\rceil}
\def\Char{\operatorname{char}}

\def\cl{\operatorname{cl}}

\def\depth{\operatorname{depth}}

\def\dim{\operatorname{dim}}

\def\ffc{\operatorname{ffc}}

\def\floor#1{\left\lfloor #1 \right\rfloor}

\def\KK{{\mathbb K}}

\def\reg{\operatorname{reg}}

\def\depth{\operatorname{depth}}

\def\ZZ{{\mathbb Z}}

\newcommand\bdX{\bm{X}}

\newcommand\calC{\mathcal{C}}

\newcommand\calJ{\mathcal{J}}

\begin{document}

\title{Generalized binomial edge ideals of bipartite graphs}

\author{Yi-Huang Shen}
\address{CAS Wu Wen-Tsun Key Laboratory of Mathematics, School of Mathematical Sciences, University of Science and Technology of China, Hefei, Anhui, 230026, P.R.~China}
\email{yhshen@ustc.edu.cn}

\author{Guangjun Zhu$^{\ast}$}
\address{School of Mathematical Sciences, Soochow University, Suzhou, Jiangsu, 215006, P.R.~China}
\email{zhuguangjun@suda.edu.cn}

\thanks{$^{\ast}$ Corresponding author}
\thanks{2020 {\em Mathematics Subject Classification}.
Primary 13C13, 13C15; Secondary 05E40, 13D02, 13F20}

\thanks{Keywords: Regularity, dimension, depth, generalized binomial edge ideal, bipartite graph, fan graph}

\begin{abstract}
    Connected bipartite graphs whose binomial edge ideals are Cohen--Macaulay have been classified by Bolognini et al. In this paper, we compute the depth, Castelnuovo--Mumford regularity, and dimension of the generalized binomial edge ideals of these graphs.
\end{abstract}

\maketitle

\section{Introduction}

Let $m$ and $n$ be two positive integers. Following the convention, the
notation $[m]$ denotes the set $\{1,2,\ldots, m\}$. 
Let $S=\KK[\bdX]\coloneqq \KK[x_{ij}:i\in[m],j\in[n]]$ be the polynomial ring over the field $\KK$ in $m\times n$ variables.
In \cite{MR3290687}, Ene et al.~introduced the {binomial edge ideal of a pair of graphs}. Specifically, let $G_1$ and $G_2$ be simple graphs on vertex sets $[m]$ and $[n]$ respectively. Suppose that $e=\{i, j\}\in E(G_1)$ and $f=\{t, l\}\in E(G_2)$ are two edges with $i<j$ and $k<l$. Then, one can assign a $2$-minor $p_{(e,f)}=[i,j\,|\,t,l]\coloneqq x_{it}x_{jl}-x_{il}x_{jt}$ to the pair $(e, f)$. The \emph{binomial edge ideal of the pair $(G_1, G_2)$} is defined to be
\[
    \calJ_{G_1,G_2}\coloneqq (p_{(e,f)} : e \in E(G_1), f \in E(G_2))
\]
in $S$. This is a generalization of the classical \emph{binomial edge ideals}
in \cites{MR2669070, MR2782571} when one of $G_1$ and $G_2$ is the complete
graph $K_2$. Meanwhile, the \emph{ideals generated by adjacent minors} in
\cite{MR1627343} turn out to be the binomial edge ideals of a pair of line
graphs.

Much research has been done to understand
the connection between the
algebraic properties of $\calJ_{G_1,G_2}$ and the combinatorial properties of
$G_1$ and $G_2$.
In particular, the combinatorial invariants of $G_1$ and $G_2$ can be related to homological invariants such as depth, Betti numbers, and Castelnuovo--Mumford regularity (regularity for short) of $\calJ_{G_1,G_2}$, see \cites{MR3290687, MR3040610}. In general, however,  the algebraic properties of these ideals are still widely open. For $S/\calJ_{G_1,G_2}$ possessing a decent algebraic property, a good rule of thumb is that one of $G_1$ and $G_2$ is a complete graph. For example, it was proved in \cite[Theorem 1.2, Corollary 2.2 and Proposition 4.1]{MR3290687} that this is an equivalent condition for $\calJ_{G_1,G_2}$ to be radical ideal. Therefore, $\calJ_{G_1,G_2}$ is a prime ideal if and only if $G_1$ and $G_2$ are complete. Furthermore, the condition for $\calJ_{G_1,G_2}$ to be unmixed is that one of the graphs is complete and the other graph satisfies certain numerical conditions for all its subsets with the cut-point property. 
The need to include complete graphs can also be justified by \cite[Theorem 3.1]{MR3859970} and \cite[Theorems 1 and 10]{MR3040610}.

Let $K_m$ be a complete graph on $m$ vertices and $G$ be a simple graph. 
Then, $\calJ_{K_m,G}$ is the \emph{generalized binomial edge ideal} associated with $G$, which was previously introduced by Rauh in \cite{MR3011436} for studying conditional independence ideals. 
Inspired by the progress regarding binomial edge ideals, researchers turn to the study of the generalized binomial edge ideals. For example,
in \cite{MR4233116}, Chaudhry and Irfan considered the depth and regularity of the generalized binomial edge ideals of connected generalized block graphs. In \cite{MR4033090}, Kumar proved that when $G$ is a 
connected graph with $n$ vertices, the regularity of $S/\calJ_{K_m,G}$ is bounded above by $n-1$. Moreover, he showed that if $m\ge n$, then the regularity of $S/\calJ_{K_m,G}$ is exactly $n-1$, and if $m<n$, then the regularity of $S/\calJ_{K_m,G}$ is bounded below by $\max\{m-1, \ell(G)\}$, where $\ell(G)$ is the length of the longest induced path in $G$. In \cite{arXiv:2207.02256}, Katsabekis studied the {cohomological dimension} of $\calJ_{K_m,G}$. For a connected graph $G$, he found out that $\calJ_{K_m,G}$ is a cohomologically complete intersection for $m\ge 3$ if and only if $G$ is the complete graph and $\Char(\KK)>0$. In this case, $\calJ_{K_m,G}$ is Cohen--Macaulay. In \cite{arXiv:2112.15136}, Amata et al.~characterized the unmixedness of the generalized binomial edge ideals associated with non–complete power cycles. Nevertheless, so far, not much is known about the algebraic properties of the generalized binomial edge ideals.

Recall that in \cite{MR3779601} Bolognini et al.~classified  the
Cohen--Macaulay binomial edge ideals of bipartite graphs, giving an explicit
and recursive construction in graph-theoretic terms. As a basic block, they
introduced a family of bipartite graphs $F_p$ whose binomial edge ideals are
Cohen--Macaulay. They also introduced an auxiliary family of non-bipartite
graphs $F_k^{W}(K_n)$, called \emph{$k$-fan graphs}, whose binomial edge
ideals are also Cohen--Macaulay. In one of the most important results, they
proved that if $G$ is a connected bipartite graph, then $J_G$ is
Cohen--Macaulay if and only if $G=G_1*\cdots * G_s$, where $G_i=F_p$ or
$G_i=F_{p_1}\circ\cdots \circ F_{p_t}$ for some $p\ge 1$ and $p_j\ge 3$; see
\Cref{sec:prelim} for the definition of the operations $*$ and $\circ$.
Later, in \cite{MR3991052}, Jayanthan et al.~ studied the regularity of
binomial edge ideals of these graphs by understanding the behavior of the
regularity under the operation $\circ$. In this paper, we are mainly
interested in the depth and Castelnuovo--Mumford regularity of the \emph{generalized}
binomial edge ideal associated with these graphs. Unfortunately, the $*$
operation does not behave as nicely as in \cite{MR3991052}. We compensate for this
by including computations of local cohomology modules and  repetitive 
glue-and-delete operations on related graphs.

The article is organized as follows. In Section \ref{sec:prelim}, we recall some essential definitions and terminology that we will need later.  In Sections \ref{sec:fan_graph} and \ref{sec:F_p}, we study the dimension, depth, and regularity of the generalized binomial edge ideal of the fan graph $F_k^W(K_n)$ and the bipartite graph $F_p$ respectively. We then introduce a class of simple graphs by gluing these two types of graphs together, via the $*$ and $\circ$ operations. 
Some formulas for the depth, regularity, and dimension of the generalized binomial edge ideals within this class are given
in Sections \ref{sec:depth}, \ref{sec:regularity}, and \ref{sec:dim}, respectively.

\section{Preliminaries}
\label{sec:prelim}

In this section, we collect definitions and basic facts that will be used throughout this paper. For more details, the reader is referred to \cites{MR4423525, MR3941158, MR3991052}.

\subsection{Regularity and depth}
In the sequel, let $S_+$ be the unique graded maximal ideal of the standard
graded algebra $S$. The local cohomology modules of a finitely generated
graded $S$-module $M$ with respect to $S_+$ are denoted by $H_{S_+}^i(M)$
for $i\in \ZZ$.

\begin{Definition}
    Let $M$ be a finitely generated graded $S$-module.
    \begin{enumerate}[a]
        \item The \emph{depth} of $M$ is defined as
            \[
                \depth\,(M) \coloneqq \min\{i: H_{S_+}^i(M)\ne 0\}.
            \]
        \item For $i = 0, \dots , \dim(M)$, the $i$\textsuperscript{th} \emph{$a$-invariant} of $M$ is defined as
            \[
                a_i(M) \coloneqq \max\{t : (H_{S_+}^i(M))_t \ne 0\}
            \]
            with the convention that $\max \emptyset = -\infty$.
        \item The \emph{Castelnuovo--Mumford regularity} of $M$ is defined as
            \[
                \reg(M) \coloneqq \max\{a_i(M) + i:0\le i\le \dim(M)\}.
            \]
    \end{enumerate}
\end{Definition}

The following lemmas are often used to compute the depth and regularity of a module.
In particular, since the facts in \Cref{lem:direct_sum} are well-known, they will be used implicitly in this paper.

\begin{Lemma}
    \label{lem:direct_sum}
    Let $M,N$ be two finitely generated graded $S$-modules. Then,
    $\depth(M\oplus N)=\min\{\depth(M),\depth(N)\}$ and
    $\reg(M\oplus N)=\max\{\reg(M),\reg(N)\}$.
\end{Lemma}

\begin{Lemma}
    [{\cite[Lemmas 2.1 and 3.1]{MR2643966}}]
    \label{depthlemma}
    Let $0\rightarrow M \rightarrow N \rightarrow P \rightarrow 0$ be a short
    exact sequence of finitely generated graded $S$-modules.
    \begin{enumerate}[a]
        \item \label{depthlemma-a} One has $\depth(M)\ge \min \{\depth(N),
            \depth(P)+1\}$. The equality holds if $\depth(N)\neq\depth(P)$.
        \item \label{depthlemma-c} One also has
            $\reg(M)\le\max\{\reg(N),\reg(P)+1\}$. The equality holds if
            $\reg(N)\neq\reg(P)$.
    \end{enumerate}
\end{Lemma}

 \subsection{Notions of simple graphs}

Let $G$ be a simple graph with the vertex set $V(G)$ and the edge set $E(G)$. Let $\widetilde{G}$ denote the complete graph on $V(G)$.
For any subset $A$ of $V(G)$, let $G[A]$ denote the \emph{induced subgraph} of $G$ on
the vertex set $A$, i.e., for $i,j \in A$, $\{i,j\} \in E(G[A])$ if and only
if $\{i,j\}\in E(G)$. At the same time, we denote the induced subgraph of $G$
on $V(G)\setminus A$ by $G\setminus A$.

If $v$ is a vertex in $G$, its \emph{neighborhood} is defined as $N_G(v):=
\{u \in V(G): \{u,v\}\in E(G)\}$.
Let $N_G[v]\coloneqq N_G(v)\cup \{v\}$. The vertex $v$ is called a \emph{free vertex} or \emph{leaf} if the cardinality $|N_G(v)|=1$;
otherwise, it is called an \emph{internal vertex}. We will write $G\setminus
v$ instead of $G\setminus \{v\}$ for simplicity. Meanwhile, for this vertex,
let $G_v$ denote the graph on the vertex set $V(G)$ with edge set
$E(G_v)=E(G)\cup \{\{u,w\} : u,w \in N_G (v)\}$.

Let $c(G)$ denote the number of connected components of $G$. A vertex $v$ is
called a \emph{cut vertex} of $G$ if $c(G)< c(G\setminus v)$. Let $T$ be a
subset of $V(G)$.
As an abuse of notation, we also let $c(T)$ denote the number of connected
components of $G\setminus T$.
If $v$ is a cut vertex of the induced subgraph $G\setminus (T\setminus \{v\})$
for any $v\in T$, then we say that $T$ has the \emph{cut point property}. Set
$ \calC(G) \coloneqq \{\emptyset\}\cup\Set{T: T \text{\ has the cut point
property}}$.

Finally, the graph $G$ is said to be \emph{bipartite} if there exists a decomposition
of $V(G)=V_1\sqcup V_2$ into two disjoint subsets, such
that no two vertices of the same subset are adjacent in $G$.

\subsection{Basic facts regarding generalized binomial ideals}
Throughout what follows, we always assume that $m,n\ge 2$ are two positive
integers.  The following result is well-known.

\begin{Lemma}
    [{\cite[Corollary 4]{MR266912}}]
    \label{lem:generic_CM}
    Let $\bdX=(x_{ij})_{1\le i\le m,1\le j\le n}$ be the matrix of
    indeterminates over $\KK$. Then the ideal $I_t(\bdX)$ in $S=\KK[x_{ij}:i\in [m]
    \text{ and } j\in [n]]$, generated by all $t\times t$ minors of $\bdX$,
    is perfect of height $(m-t+1)(n-t+1)$. In particular,
    $S/\calJ_{K_m,K_n}=S/I_2(\bdX)$ is Cohen--Macaulay of dimension $
    m+n-1$.
\end{Lemma}

Suppose that $G$ is a simple graph on $[n]$. From now on, we will study its
generalized binomial edge ideal $\calJ_{K_m,G}$ in $S=\KK[x_{ij}:i\in [m]
\text{ and } j\in [n]]$. We also write this ring as $S_G$ when we want to emphasize
its relation to the graph $G$.

For each subset $T$ of $[n]$, we can introduce the ideal
\[
    P_T(K_m, G)\coloneqq (x_{ij}: (i,j)\in [m]\times
    T)+\calJ_{K_m,\widetilde{G_1}}+\cdots+\calJ_{K_m,\widetilde{G_{c(T)}}}
\]
in $S$, where $G_1,\ldots, G_{c(T)}$ are the connected components of
$G\setminus T$. It is clear from \Cref{lem:generic_CM} that
\begin{align}
    \dim(S/P_T(K_m, G)) &
    =\sum_{t=1}^{c(T)}(m+|V(\widetilde{G_i})|-1)=c(T)(m-1)+
    \sum_{t=1}^{c(T)}|V(\widetilde{G}_i)|\notag
    \\
    & =c(T)(m-1)+|V(G)|-|T|. \label{eqn:dim_prime}
\end{align}
The minimal primes of $\calJ_{K_m,G}$ can be described explicitly by these
ideals.

\begin{Lemma}
    [{\cite[Theorem 7]{MR3011436}}]
    \label{lem:decompo}
    Let $G$ be a finite simple graph. Then,
    \[
        \calJ_{K_m,G}
        =\bigcap_{T\in \calC(G)}P_T(K_m, G)
    \]
    is the minimal primary decomposition of the radical ideal $\calJ_{K_m,G}$.
\end{Lemma}

The theoretical argument of this paper depends essentially on the study of the
following short exact sequence.

\begin{Lemma}
    [{\cite[Theorem 3.2]{MR4033090}}]
    \label{rem:exactseq}
    Let $G$ be a finite simple graph and $v$ be an internal vertex in $G$.
    Then
    \[
        \calJ_{K_m,G}=\calJ_{K_m,G_v}\cap ((x_{iv}: i\in [m])+\calJ_{K_m,G\setminus v}).
    \]
    Note that
    \[
        \calJ_{K_m,G_v}+((x_{iv}: i\in [m])+\calJ_{K_m,G\setminus v})=(x_{iv}:
        i\in [m])+\calJ_{K_m,G_v\setminus v}.
    \]
    Hence, we have the following short exact sequence:
    \[
        0\to \frac{S}{\calJ_{K_m,G}} \to \frac{S}{\calJ_{K_m,G_v}}\oplus
        \frac{S}{(x_{iv}: i\in [m])+\calJ_{K_m,G\setminus v}}\to
        \frac{S}{(x_{iv}: i\in [m])+\calJ_{K_m,G_v\setminus v}}\to 0,
    \]
    namely,
    \begin{equation}
        0\to \frac{S_G}{\calJ_{K_m,G}} \to
        \frac{S_{G_v}}{\calJ_{K_m,G_v}}\oplus \frac{S_{G\setminus
        v}}{\calJ_{K_m,G\setminus v}}\to \frac{S_{G_v\setminus
        v}}{\calJ_{K_m,G_v\setminus v}}\to 0.
        \label{eqn:SES-1}
    \end{equation}
\end{Lemma}

We end this section by recalling several key results for understanding the
regularity of generalized binomial edge ideals.

\begin{Lemma}
    [{\cite[Proposition 8]{MR3040610}}]
    \label{cor:induced_graph}
    Let $G$ be a graph over $[n]$ and $H$ be an induced subgraph of $G$. Then,
    we have $\reg(S/\mathcal{J}_{K_m,H}) \le \reg(S/\mathcal{J}_{K_m,G})$.
\end{Lemma}

\begin{Corollary}
    \label{lem:reg_easy_bound}
    Let $G$ be a simple graph and $v$ be an internal vertex of $G$. Then,
    \[
        \reg(S/\calJ_{K_m,G})\le
        \max\{\reg(S/\calJ_{K_m,G_v})+1,\reg(S/\calJ_{K_m,G\setminus v})\}.
    \]
\end{Corollary}
\begin{proof}
    Note that $G_v\setminus v$ is an induced subgraph of $G_v$. So
    $\reg(S/\calJ_{K_m,G_v\setminus v})\le \reg(S/\calJ_{K_m,G_v})$ by
    \Cref{cor:induced_graph}. Now, it remains to apply
    \Cref{depthlemma}\ref{depthlemma-c} to the exact sequence
    (\ref{eqn:SES-1}) in \Cref{rem:exactseq}.
\end{proof}

\begin{Lemma}
    [{\cite[Theorems 3.6 and 3.7]{MR4033090}}]
    \label{lem:reg_min_equal}
    Let $G$ be a connected graph.
    Then $\reg(S/\calJ_{K_m,G})\le
    |V(G)|-1$.
    If additionally $m\ge |V(G)| \ge 2$, then
    $\reg(S/\calJ_{K_m,G})=|V(G)|-1$.
    Moreover, $\reg(S/\calJ_{K_m,K_n})
    =\min\{m-1,n-1\}$.
\end{Lemma}

\section{Study of fan graphs}
\label{sec:fan_graph}

In order to study bipartite graphs whose binomial edge ideals are Cohen--Macaulay, Bolognini et al.~in \cite{MR3779601} introduced the fan graphs of complete graphs, which is a family of chordal graphs. In this section, we will study the dimension, depth, and regularity of the quotient ring of the generalized binomial edge ideal of these graphs. Let us start with the definition reformulated in \cite{MR3991052}.

\begin{Definition}
    [{\cite[Definition 3.1]{MR3991052}}]
    Let $K_n$ be a complete graph on the vertex set $[n]$.
    \begin{enumerate}[a]
        \item Let $U=\{u_1,u_2,\ldots,u_r\}$ be a subset of $[n]$. Suppose that for each $i\in [r]$, a new complete graph $K_{a_i}$ with $a_i>i$ is attached to $K_n$ in such a way that $V(K_n)\cap V(K_{a_i})= \{u_1,u_2, \dots,u_i\}$. We say that the resulting graph is obtained by \emph{adding a fan} to $K_n$ on the set $U$ and $\{K_{a_1},\ldots,K_{a_r}\}$ is the \emph{branch} of that fan on $U$.

        \item Let $W$ be a subset of $[n]$ and suppose that $W=W_1\sqcup\cdots\sqcup W_k$ is a partition of $W$. Let $F^W_k(K_n)$ be a graph obtained from $K_n$ by adding a fan on each  $W_i$. The resulting graph $F^W_k(K_n)$ is called a \emph{$k$-fan graph} of $K_n$ on the set $W$. For future reference, for each $i\in [k]$, we assume that $\{K_{a_{i,1}},\ldots,K_{a_{i,r_i}}\}$ is the branch of the fan on $W_i=\{w_{i,1},\ldots,w_{i,r_i}\}$. For notational convenience, we also set $h_{i,j}\coloneqq a_{i,j}-j$.

        \item A branch $\{K_{a_{i,1}},\ldots,K_{a_{i,r_i}}\}$ of the fan on $W_i$ is  called a \emph{pure branch} if $h_{i,j}=1$ for every $j\in [r_i]$. Furthermore, if each branch is pure, then $F^W_k(K_n)$ is said to be a \emph{$k$-pure fan graph} of $K_n$ on $W$.

        \item The complete graph $K_n$ can be considered as a degenerate fan without branches, i.e., with $k=0$.
    \end{enumerate}
\end{Definition}

The following result generalizes the dimensional formula contained in the proof of \cite[Lemma 3.2]{MR3779601} for the case $m=2$.

\begin{Theorem}
    \label{thm:fan_Dimension}
    Let $G=F_k^W(K_n)$ be a $k$-fan graph with $n\ge 2$.
    Then, we have
    \[
        \dim(S/\calJ_{K_m,G})=m+|V(G)|-1+s(m-2)
    \]
    where $s=\min\{|W|,n-1\}$.
\end{Theorem}

\begin{proof}
    When $k=0$, the set $W=\emptyset$ and $G=K_n$ is a complete graph. The
    claimed equality is clear in this case by \Cref{lem:generic_CM} .

    In the following, we assume that $k\ge 1$.  For each $i\in [k]$, with respect
    to $W_i=\{w_{i,1},\ldots,w_{i,r_i}\}$, we set $\calC_i\coloneqq
    \{\emptyset\}\cup \left\{\{w_{i,1},\ldots,w_{i,j}\}: j\in [r_i]\right\}$.
    From the proof of \cite[Lemma 3.2]{MR3779601}, we know that the collection
    of sets with cut point property is
    \[
        \calC(G)=\{T_1\sqcup\cdots \sqcup T_k: T_i\in \calC_i\text{\ and\ }
        T_1\sqcup\cdots\sqcup T_k\subsetneq [n]\}.
    \]
    For any $T\in \calC(G)$, we have the corresponding decomposition $T=
    T_1\sqcup\cdots\sqcup T_k$ with $T_i\in \calC_i$. It is clear that
    $G\setminus T$ consists of the following $|T|+1$ components:
    \begin{enumerate}[a]
        \item the complete graphs $K_{h_{i,j}}$ for each $i\in [k]$ and $j\in [|T_i|]$, and
        \item the fan graph $G'$ of $K_n\setminus T$ on the set $\bigsqcup_i (W_i\setminus T_i)$.
    \end{enumerate}
    It follows from equation \eqref{eqn:dim_prime} that
    \begin{align*}
        \dim \left(S/P_T(K_m, G)\right) & = (|T|+1)(m-1)+|V(G)|-|T|=m+|V(G)|-1+|T|(m-2).
    \end{align*}
    Hence  by \Cref{lem:decompo} we get
    \begin{align*}
        \dim(S/\calJ_{K_m,G}) & =\max\{m+|V(G)|-1+|T|(m-2): T\in \calC(G)\} \\
        & =m+|V(G)|-1+s(m-2)
    \end{align*}
    where $s=\min\{|W|,n-1\}$.
\end{proof}

Next, we study the depth and regularity of the generalized binomial edge ideal of fan graphs in a fairly general situation.

\begin{Theorem}
    \label{thm:fan_general}
    Let $G=F_k^W(K_n)$ be a $k$-fan graph with $n\ge 2$. Let
    $\delta\coloneqq  |\{(i,j):h_{i,j}\ge 2\}|$. Then, we have the following.
    \begin{enumerate}[a]
        \item \label{thm:fan_general_a}
            The depth of $S/\calJ_{K_m,G}$ is given by $|V(G)|+m-1$.
        \item \label{thm:fan_general_b}
            If $m \ge |V(G)|$, then $\reg(S/\calJ_{K_m,G})=|V(G)|-1$.
        \item \label{thm:fan_general_c}
            If $m< |V(G)|$, then
            $\reg(S/\calJ_{K_m,G})\le k+(\delta+1)(m-1)$.
        \item \label{thm:fan_general_d}
            We have $\reg(S/\calJ_{K_m,G})\le \cl(G)(m-1)$, where
            $\cl(G)$ is the number of maximal cliques of $G$. Unless $G=F_1^W(K_n)$ with $W=[n]$, this number of cliques is exactly  $|W|+1$.
        \item \label{thm:fan_general_e}
            Suppose further that  
            $n>1+|W|$ and $h_{i,j}\ge m$ for all $i\in [k]$ and $j\in [r_i]$.
            Then
            $\reg(S/\calJ_{K_m,G})= (\cl(G)-1)(m-1)+\min\{m-1,n-|W|-1\}$.
    \end{enumerate}
\end{Theorem}

\begin{proof}
    \begin{enumerate}[a]
        \item \label{thm:fan_general_a_proof}
            We prove by induction on $N:=\sum_{i=1}^k\sum_{j=1}^{r_i}
            h_{i,j}$. The base case is when $N=0$. Whence, $G=K_n$ and
            $k=\delta=0$.  It follows from \Cref{lem:generic_CM} that
            $\depth(S/\calJ_{K_m,G})=m+n-1$ in this case.

            In the following, we assume that $N\ge 1$. Whence, $k\ge 1$ as well.
            With respect to $v=w_{k,1}\in W_k$, we note that $G_v$ (resp.
            $G_v\setminus v$) is a $(k-1)$-fan of
            $K_{n+h_{k,1}+\cdots+h_{k,r_k}}$  (resp.
            $K_{n+h_{k,1}+\cdots+h_{k,r_k}-1}$) on the set $W\setminus W_k$.
            Thus, by induction, we have
            \begin{align*}
                \depth(S_{G_v}/\calJ_{K_m, G_v}) & =m+|V(G)|-1, \\
                \intertext{and}
                \depth\left(S_{G_v\setminus v}/\calJ_{K_m,G_v\setminus v}\right) & =m+|V(G)|-2.
            \end{align*}
            Meanwhile, $G\setminus v$ is the disjoint union of $K_{h_{k,1}}$
            and $G'$, where $G'$ is a $k$-fan graph of $K_{n-1}$ on the set
            $W\setminus v$ if $r_{k}\ge 2$, or a $(k-1)$-fan of $K_{n-1}$ on
            the set $W\setminus v$ if $r_k=1$. Thus, by induction and
            \Cref{lem:generic_CM}, we have
            \begin{align*}
                \depth\left(S_{G\setminus v}/\calJ_{K_m,G\setminus v}\right) & =(m+h_{k,1}-1)+(m+|V(G)|-h_{k,1}-2) \\
                & =m+|V(G)|-1+(m-2).
            \end{align*}
            Consequently, by applying
            \Cref{depthlemma}\ref{depthlemma-a} to the short exact sequence
            (\ref{eqn:SES-1}) in \Cref{rem:exactseq}, we obtain the claimed
            formula in this case.

        \item The assertion in \ref{thm:fan_general_b} follows from
            \Cref{lem:reg_min_equal}.

        \item \label{thm:fan_general_c_proof}
            In this case, we have $m< |V(G)|$ from the assumption. We
            still prove by induction on the $N$ defined in the part
            \ref{thm:fan_general_a_proof} of the proof and argue in the same way.

            The base case is when $N=0$.  Whence, $G=K_n$ and $k=\delta=0$. It
            follows from \Cref{lem:reg_min_equal} that
            $\reg(S/\calJ_{K_m,K_n})=\min\{m-1,n-1\}= m-1$, establishing the
            claimed inequality in this case.

            In the following, we assume that $N\ge 1$.  Whence, $k\ge 1$ as well.
            We still take $v=w_{k,1}\in W_k$. Using the analysis in part
            \ref{thm:fan_general_a_proof} of the proof, we have  by induction 
            \begin{align}
                \reg(S_{G_v}/\calJ_{K_m,G_v}) &\le k-1+(\delta'+1)(m-1)
                \label{eqn:c_4}
                \intertext{where $\delta':=|\{(i,j):h_{i,j}\ge 2 \text{ and }
                i<k\}| \le \delta$, and}
                \reg\left(S_{G\setminus v}/\calJ_{K_m,G\setminus v}\right)  &
                \le  \min\{h_{k,1}-1,m-1\}+ k+(\delta''+1)(m-1)
                \label{eqn:c_5}
            \end{align}
            where $\delta''\coloneqq |\{(i,j):h_{i,j}\ge 2 \text{ and } (i,j)\ne (k,1)\}| \le \delta$.  Note that the RHS of \eqref{eqn:c_4} $\le k-1+(\delta+1)(m-1)$.  Meanwhile, if $h_{k,1}=1$, then $\delta''=\delta$ and the RHS of \eqref{eqn:c_5} $=k+(\delta+1)(m-1)$. On the other hand, if $h_{k,1}>1$, then $\delta''=\delta-1$ and the RHS of \eqref{eqn:c_5} $\le (m-1)+k+\delta(m-1)=k+(\delta+1)(m-1)$. As a result, by applying \Cref{lem:reg_easy_bound}, we obtain the expected inequality in this case.

        \item \label{thm:fan_general_d_proof}
            The proof here is similar to case \ref{thm:fan_general_c_proof}. 
            We continue to prove by induction on the $N$ defined in part \ref{thm:fan_general_a_proof} of the proof, and argue in the same way.
            In the base case, $N=0$.  Whence, $G=K_n$ and $\cl(G)=1$. It follows from \Cref{lem:reg_min_equal} that $\reg(S/\calJ_{K_m,G})=\min\{m-1,n-1\}\le \cl(G)(m-1)$, which gives
            the desired inequality in the base case.

            For the induction step, we assume that $N\ge 1$. By induction,
            for $v=w_{k,1}\in W_k$, we have
            \begin{align*}
                \reg(S_{G_v}/\calJ_{K_m,G_v})  & \le
                \cl(G_v)(m-1)=(\cl(G)-|W_k|)(m-1).
            \end{align*}
            Meanwhile,  we have
            \begin{align*}
                \reg(S_{G\setminus v}/\calJ_{K_m,G\setminus v}) &
                =\reg(S/\calJ_{K_m,G'})+\reg(S/\calJ_{K_m,K_{h_{k,1}}}) \\
                & \le (\cl(G)-1)(m-1)+\min\{h_{k,1}-1,m-1\}\\
                &\le \cl(G)(m-1).
            \end{align*}
            Hence, by
            applying \Cref{lem:reg_easy_bound},
            we obtain 
            \begin{align*}
                \reg(S/\calJ_{K_m,G}) & \le
                \max\{(\cl(G)-|W_k|)(m-1)+1,\cl(G)(m-1)\} \\
                & =\cl(G)(m-1),
            \end{align*}
            since $|W_k|\ge 1$ and $m\ge 2$. Consequently, we get the
            expected inequality in this case.

        \item We continue the proof by induction on the $N$ defined in part \ref{thm:fan_general_a_proof} of the proof, and argue in the same way as before.
            When $N=0$ in the base case, $\cl(G)=1$ and $|W|=0$. 
            The desired equality obviously exists. For the induction step,
            since the corresponding assumptions for $G'$ are still satisfied,
            we have
            \begin{align*}
                \reg(S_{G\setminus v}/\calJ_{K_m,G\setminus v}) &
                =\reg(S/\calJ_{K_m,G'})+\reg(S/\calJ_{K_m,K_{h_{k,1}}})
                \\
                & =(\cl(G)-2)(m-1)+\min\{m-1,n-1-(|W|-1)-1\}\\
                & \qquad +\min\{h_{k,1}-1,m-1\} \\
                & =(\cl(G)-1)(m-1)+\min\{m-1,n-|W|-1\}.
            \end{align*}
            Since $m\ge 2$ and $r_k\ge 1$, from the proof of in 
            part \ref{thm:fan_general_d_proof}, we have
            \begin{align*}
                \reg(S_{G_v}/\calJ_{K_m,G_v})&\le (\cl(G)-|W_k|)(m-1)\\
                &< (\cl(G)-1)(m-1)+\min\{m-1,n-|W|-1\}\\
                & =\reg(S_{G\setminus v}/\calJ_{K_m,G\setminus v}).
            \end{align*}
            So, according to \Cref{lem:reg_easy_bound}, we have
            \begin{align*}
                \reg(S/\calJ_{K_m,G}) &=\reg(S_{G\setminus
                v}/\calJ_{K_m,G\setminus v})\\
                & =(\cl(G)-1)(m-1)+\min\{m-1,n-|W|-1\}.
            \end{align*}
            This gives the expected inequality in this case.
            \qedhere
    \end{enumerate}
\end{proof}

We can draw two quick conclusions from the previous arguments.

\begin{Corollary}
    Let $G=F_k^W(K_n)$ be a $k$-fan graph with $n\ge 2$.
    Then the following are equivalent:
    \begin{enumerate}[a]
        \item the ring $S/\calJ_{K_m,G}$ is Cohen--Macaulay;
        \item the ideal $\calJ_{K_m,G}$ is unmixed;
        \item either $m=2$ or $W=\emptyset$.
    \end{enumerate}
\end{Corollary}
\begin{proof}
    This is clear from the proofs of \Cref{thm:fan_Dimension} and
    \Cref{thm:fan_general}\ref{thm:fan_general_a}.
\end{proof}

\begin{Corollary}
    \label{cor:FkW_pure}
    Let $G=F_k^W(K_n)$ be a $k$-pure fan  graph   with $n\ge 2$. Then
    $\reg(S/\calJ_{K_m,G})\le k+m-1$.
\end{Corollary}
\begin{proof}
    Since $G$ is a pure fan, the $\delta$ defined in \Cref{thm:fan_general} is
    exactly $0$. Thus, the claimed inequality follows from
    \Cref{thm:fan_general} \ref{thm:fan_general_b} and \ref{thm:fan_general_c}.
\end{proof}

In contrast to  \cite[Theorem 3.4]{MR3991052}, we don't have equality  in general in \Cref{cor:FkW_pure}.
The exact formula for $\reg(S/\calJ_{K_m,G})$ will be presented in \Cref{thm:FkW_pure_reg}. To prove it, we need to make some preparations.

\begin{Observation}
    \label{obs:fan_graph}
    Let $G=F_k^W(K_n)$ be a $k$-pure fan graph. Then $|V(G)|=n+|W|$. Suppose that this is a fan graph with respect to $W=W_1\sqcup \cdots \sqcup W_k$, where $W_i=\{w_{i,1},\dots,w_{i,r_i}\}$. Now take $v=w_{k,1}$. Then $G_v$ can be thought of as $F_{k-1}^{W\setminus W_k}(K_{n+r_k})$ and the graph $G_v\setminus v$ can be thought of as $F_{k-1}^{W\setminus W_k}(K_{n+r_k-1})$. Furthermore, the graph $G\setminus v$ can be thought of as the disjoint union of a fan graph $G'$ with an isolated vertex. If $r_k=1$, then $G'\coloneqq F_{k-1}^{W\setminus \{v\}}(K_{n-1})$. If instead $r_k\ge 2$, then $G'\coloneqq F_{k}^{W\setminus \{v\}}(K_{n-1})$.
\end{Observation}

\begin{Lemma}
    \label{lem:fan_kn_v<m+k}
    Let $G=F_k^W(K_n)$ be a $k$-pure fan graph. Then we have the following
    results.
    \begin{enumerate}[a]
        \item \label{lem:fan_kn_v<m+k_a}
            If $i\le |V(G)|+m-2$, then the local cohomology module
            $H_{S_+}^{i}(S/\calJ_{K_m,G})=0$.
        \item \label{lem:fan_kn_v<m+k_b}
            If $|V(G)|\le m+k$, then $a_{|V(G)|+m-1}(S/\calJ_{K_m,G})\ge -m$.
        \item \label{lem:fan_kn_v<m+k_c}
            If $|V(G)|\ge m+k$, then $a_{|V(G)|+m-1}(S/\calJ_{K_m,G})\ge
            -|V(G)|+k$.
    \end{enumerate}
\end{Lemma}

\begin{proof}
    The nullity in \ref{lem:fan_kn_v<m+k_a} follows from the fact that the $\depth\,(S/\calJ_{K_m,G})=|V(G)|+m-1$ in \Cref{thm:fan_general}\ref{thm:fan_general_a} and the graded version of \cite[Theorem 6.2.7]{MR1613627} or \cite[Proposition 3.5.4 (b)]{MR1251956}.

    For the remaining \ref{lem:fan_kn_v<m+k_b} and \ref{lem:fan_kn_v<m+k_c}, we prove by induction on $|W|$. In the base case  we have $|W|=0$. Whence, $k=0$, $G=K_n$ and $|V(G)|=n$. The quotient ring $S/\calJ_{K_m,G}$ is Cohen--Macaulay of depth $m+n-1$ by \Cref{lem:generic_CM}. It has the  regularity $\min\{n-1,m-1\}$ by \Cref{lem:reg_min_equal}.  As $\reg(S/\calJ_{K_m,G})=a_{n+m-1}(S/\calJ_{K_m,G})+n+m-1$ by the Cohen--Macaulayness, the inequality in \ref{lem:fan_kn_v<m+k_b} holds when $n\le m$. For the same reason, the inequality in \ref{lem:fan_kn_v<m+k_c} also holds when $n\ge m$.

    In the induction step, we assume that $|W|\ge 1$. Whence, $k\ge 1$.
    Without loss of generality, we assume that
    $W_k=\{w_{k,1},\dots,w_{k,r_k}\}$ is a decomposition component of $W$. We
    apply the facts in \Cref{obs:fan_graph} with respect to $v=w_{k,1}$.
    Thus,  $H_{S_+}^{|V(G)|+m-2}(S_{G_v}/\calJ_{K_m, G_v})=0$  by \Cref{thm:fan_general}\ref{thm:fan_general_a}.
    Furthermore, $H_{S_+}^{|V(G)|+m-2}(S_{G\setminus v}/\calJ_{K_m, G\setminus
    v})=H_{S_+}^{|V(G)|-2}(S_{G'}/\calJ_{K_m, G'})=0$, since $m\ge 2$.
    As a result, we have
    \[
        0\to H_{S_+}^{|V(G)|+m-2}(S'/\calJ_{K_m,G_v\setminus v}) \to
        H_{S_+}^{|V(G)|+m-1}(S/\calJ_{K_m,G})
    \]
    from the long exact sequence of local cohomology modules induced from the
    short sequence \eqref{eqn:SES-1} in \Cref{rem:exactseq}. In particular,
    \begin{equation}
        a_{|V(G)|+m-1}(S/\calJ_{K_m,G})\ge a_{|V(G)|+m-2}( S_{G_v\setminus
        v}/\calJ_{K_m,G_v\setminus v}).
        \label{eqn:compare_a_inv_general}
    \end{equation}
    Note that by induction, if $|V(G)|\le m+k$, then
    \[
        a_{|V(G)|+m-2}(S_{G_v\setminus v}/\calJ_{K_m,G_v\setminus v})\ge
        -m
    \]
    If instead $|V(G)|\ge m+k$, then
    \[
        a_{|V(G)|+m-2}(S_{G_v\setminus v}/\calJ_{K_m,G_v\setminus v})\ge
        -(|V(G)|-1)+(k-1).
    \]
    Combining them with \eqref{eqn:compare_a_inv_general} gives the desired inequality in each case.
\end{proof}

We first determine the regularity in the baby case, where we add whiskers to a complete graph.

\begin{Proposition}
    \label{prop:FkW_pure}
    Let $G=F_k^W(K_n)$ be a $k$-pure fan graph with $r_1=\cdots=r_k=1$. Then, $\reg(S/\calJ_{K_m,G})=\min\{m,n\}+k-1$.
\end{Proposition}

\begin{proof}
    Depending on whether $n\ge m$, we have two cases.
    \begin{enumerate}[a]
        \item \label{prop:FkW_pure_a}
            Suppose that $n\ge m$. In this case, we prove that
            $\reg(S/\calJ_{K_m,G})=m+k-1$. In view of \Cref{cor:FkW_pure}, it
            suffices to show that $\reg(S/\calJ_{K_m,G})\ge m+k-1$.

            In the first subcase, we assume that $m\ge k$.  The graph $G$
            clearly has an induced subgraph $G'$, which is a $k$-pure fan
            graph of $K_m$ on $W$.  It follows from
            \Cref{lem:fan_kn_v<m+k}\ref{lem:fan_kn_v<m+k_b} that
            $a_{(m+k)+m-1}(S/\calJ_{K_m,G'})\ge -m$. Hence
            \[
                \reg(S/\calJ_{K_m,G'})\ge
                a_{(m+k)+m-1}(S/\calJ_{K_m,G'})+(m+k)+m-1\ge m+k-1.
            \]
            As a result,
            $\reg(S/\calJ_{K_m,G})\ge m+k-1$ by \Cref{cor:induced_graph}.

            In the second subcase, we assume that $m\le k$. Since $k\le n$,
            the graph $G$ has an induced subgraph $G'$, which is a $k$-pure
            fan graph of $K_k$ on $W$. It follows from
            \Cref{lem:fan_kn_v<m+k}\ref{lem:fan_kn_v<m+k_c} that
            $a_{k+m-1+k}(S/\calJ_{K_m,G'})\ge -k$. Hence
            \[
                \reg(S/\calJ_{K_m,G'})\ge
                a_{k+m-1+k}(S/\calJ_{K_m,G'})+(k+m-1+k)\ge m+k-1.
            \]
            As a result, we
            still have $\reg(S/\calJ_{K_m,G})\ge m+k-1$ by
            \Cref{cor:induced_graph}.

        \item Suppose instead that $n\le m$. In this case, we show that
            $\reg(S/\calJ_{K_m,G})=n+k-1$.  We will prove by induction on $k$.
            The base case is when $k=0$. Since the regularity is now
            $\min\{n-1,m-1\}=n-1$ by \Cref{lem:reg_min_equal}, the assertion
            holds in this case.

            In the following induction step, we assume that $k\ge 1$.  If
            $n=m$, then the assertion follows from the case \ref{prop:FkW_pure_a}.
            Therefore,  in the following, we will also assume that $n<m$. Applying
            the facts in \Cref{obs:fan_graph} by taking $v=w_{k,1}$, we have
            from induction that
            \begin{align*}
                \reg(S_{G_v}/\calJ_{K_m,G_v})& =(n+1)+(k-1)-1=n+k-1, \\
                \reg( S_{G_v\setminus v}/\calJ_{K_m,G_v\setminus v})&=n+(k-1)-1=n+k-2, \\
                \intertext{and}
                \reg(S_{G\setminus v}/\calJ_{K_m,G\setminus
                v})& =(n-1)+(k-1)-1=n+k-3.
            \end{align*}
            Applying   \Cref{depthlemma}\ref{depthlemma-c} to the short exact
            sequence \eqref{eqn:SES-1} in \Cref{rem:exactseq}, we can obtain
            \[
                \reg(S/\calJ_{K_m,G})=\max\{\max\{n+k-1,n+k-3\},n+k-2+1\}=n+k-1.
            \]
            This completes our induction argument. \qedhere
    \end{enumerate}
\end{proof}

Now we are ready to determine the regularity of the generalized binomial edge ideal of the pure fan graph.

\begin{Theorem}
    \label{thm:FkW_pure_reg}
    Let $G=F_k^W(K_n)$ be a $k$-pure fan graph. Then, we have $\reg(S/\calJ_{K_m,G})=\min\{|V(G)|-1,m+k-1\}$.
\end{Theorem}

\begin{proof}
    We prove by induction on $|W|-k$. The base case is when $|W|=k$. Whence
    $r_1=\cdots=r_k=1$. We can then apply \Cref{prop:FkW_pure} since
    $|V(G)|=n+k$.

    In the induction step, we assume that $|W| >k$. Without loss of
    generality, we may assume that $r_k\ge 2$. Whence, we apply the facts in
    \Cref{obs:fan_graph} by taking $v=w_{k,1}$. By induction, we obtain that
    \begin{align*}
        \reg(S_{G_v}/\calJ_{K_m,G_v})&=\min\{|V(G)|-1,m+k-2\}, \\
        \reg(S_{G_v\setminus v}/\calJ_{K_m,G_v\setminus v})&=\min\{|V(G)|-2,m+k-2\},\\
        \intertext{and}
        \reg(S_{G\setminus v}/\calJ_{K_m,G\setminus v})&=\min\{|V(G)|-2,m+k-1\}.
    \end{align*}
    \begin{enumerate}[a]
        \item If $|V(G)|\le m+k-1$, then the three regularities above are $|V(G)|-1$, $|V(G)|-2$ and $|V(G)|-2$ respectively. Thus, after applying \Cref{depthlemma}\ref{depthlemma-c} to the short exact sequence \eqref{eqn:SES-1}  in \Cref{rem:exactseq}, we obtain that $\reg(S_G/\calJ_{K_m,G})=|V(G)|-1$, as predicted.

        \item If $|V(G)|\ge m+k+1$, then the three regularities above are $m+k-2$, $m+k-2$ and $m+k-1$ respectively. By the same reason, we deduce that $\reg(S_G/\calJ_{K_m,G})=m+k-1$ in the expected formula.

        \item In the last case, we suppose that $|V(G)|=m+k$. From
            \Cref{cor:FkW_pure} we have $\reg(S/\calJ_{K_m,G}) \le m+k-1$.
            Meanwhile, it follows from
            \Cref{lem:fan_kn_v<m+k}\ref{lem:fan_kn_v<m+k_c} that
            \[
                \reg(S_G/\calJ_{K_m,G})\ge
                a_{|V(G)|+m-1}(S/\calJ_{K_m,G})+|V(G)|+m-1\ge m+k-1.
            \]
            Therefore, the expected equality also holds in this case. This concludes our induction argument. 
            \qedhere
    \end{enumerate}
\end{proof}

\section{Study of the bipartite graph $F_p$}
\label{sec:F_p}

In this section, we will provide exact formulas for the dimension, depth and regularity of the generalized binomial edge ideal of the following bipartite graph $F_p$.

\begin{Definition}
    For an integer $p\ge 1$, let $F_p$ denote the graph on the vertex set $[2p]$ with the edge set $E(F_p)=\{\{2i,2j-1\} : 1\le i\le p \text{ and } i\le j \le p\}$.
\end{Definition}

Following the important work of Bolognini et al.~in \cite{MR3779601}, this class of graphs provides the building blocks for the bipartite graphs, whose binomial edge ideals are Cohen--Macaulay.

\begin{Observation}
    \label{obs:f_p_v}
    Let $G=F_p$ and $v=2p-1\in V(G)$. We observe the following results.
    \begin{enumerate}[i]
        \item The graph $G_v$ can be thought of as the $1$-pure fan graph $F_1^W(K_{p+1})$  where $W=V(G)\setminus N_{G}[v]$. In particular, $|W|=p-1$.
        \item The graph $G_v\setminus v$ can be thought of as the $1$-pure fan graph $F_1^W(K_{p})$ with $|W|=p-1$.
        \item The graph $G\setminus v$ can be thought of as the disjoint union of $F_{p-1}$ and the isolated vertex $2p$.
    \end{enumerate}
\end{Observation}

First, we consider the dimension and depth of the generalized binomial edge ideal of $F_p$.

\begin{Theorem}
    \label{thm:fp_dim_depth}
    Suppose that $p\ge 1$ is a positive integer. Then we have
    \begin{enumerate}[a]
        \item  \label{thm:fp_dim_depth_a}
            $\dim(S/\calJ_{K_m,F_p})=m+2p-1+(p-1)(m-2)$ and
        \item  \label{thm:fp_dim_depth_b} $\depth(S/\calJ_{K_m,F_p})=m+2p-1$.
    \end{enumerate}
    In particular, the quotient ring $S/\calJ_{K_m,F_p}$ is Cohen--Macaulay if and only if $p=1$ or $m=2$.
\end{Theorem}
\begin{proof}
    We prove the assertions by induction on $p$. The case $p=1$ follows
    from \Cref{lem:generic_CM}. Next, assume that $p\ge 2$.
    Take $v=2p-1$ as in \Cref{obs:f_p_v}. Using induction, we derive from
    \Cref{thm:fan_Dimension} and \Cref{thm:fan_general}\ref{thm:fan_general_a}
    that
    \begin{itemize}
        \item $\dim(S_{G_v}/\calJ_{K_m,G_v})=m+2p-1+(p-1)(m-2)$ and
            $\depth(S_{G_v}/\calJ_{K_m,G_v})=m+2p-1$;
        \item $\dim(S_{G_v\setminus v}/\calJ_{K_m,G_v\setminus v}) =m+2p-2+
            (p-1)(m-2)$ and $\depth(S_{G_v\setminus v}/\calJ_{K_m,G_v\setminus
            v})=m+2p-2$;
        \item $\dim(S_{G\setminus v}/\calJ_{K_m,G\setminus v})= m+(m+2p-3)
            +(p-2)(m-2)=m+2p-1+(p-1)(m-2)$ and $\depth(S_{G\setminus
            v}/\calJ_{K_m,G\setminus v})=m+m+2p-3=m+2p-1+(m-2)$.
    \end{itemize}
    Since 
    \[
        \max\{\dim(S_{G_v}/\calJ_{K_m,G_v}),\dim(S_{G\setminus v}/\calJ_{K_m,G\setminus v})\}>\dim(S_{G_v\setminus v}/\calJ_{K_m,G_v\setminus v}),
    \]
    it follows from the short exact sequence \eqref{eqn:SES-1} in \Cref{rem:exactseq} that 
    \begin{align*}
        \dim(S/\calJ_{K_m,F_p})&=\max\{\dim(S_{G_v}/\calJ_{K_m,G_v}),\dim(S_{G\setminus v}/\calJ_{K_m,G\setminus v})\}\\
        &=m+2p-1+(p-1)(m-2).
    \end{align*} 
    On the other hand, the desired assertion regarding depth can be obtained by applying \Cref{depthlemma}\ref{depthlemma-a} to the short exact sequence \eqref{eqn:SES-1} in \Cref{rem:exactseq}.
\end{proof}

Next, we compute the regularity of the generalized binomial edge ideal of the $F_p$ graph. This is a path when $p=1$ or $2$. Fortunately, the depth and regularity of generalized binomial edge ideals of paths have already been studied.

\begin{Lemma}
    [{\cite[Corollaries 3.4 and 3.6]{MR4233116}}]
    \label{lem:reg_path}
    Let $P_n$ be a path graph on the vertex set $[n]$ with $n\ge 2$. Then
    $\depth(S/\calJ_{K_m,P_n})= m+n-1$ and $\reg(S/\calJ_{K_m,P_n})= n-1$.
\end{Lemma}

As for the regularity in the general case, we need to make some preparations. 

\begin{Lemma}
    \label{prop:Fp}
    For $G=F_p$ with $p\ge 1$, we have the following results:
    \begin{enumerate}[a]
        \item \label{prop:Fp_a}
            the local cohomology module $H_{S_+}^{i}(S/\calJ_{K_m,G})=0$ for $i\le m+2p-2$;
        \item \label{prop:Fp_b}
            if $m\ge 2p-2$, then $a_{2p+m-1}(S/\calJ_{K_m,G})\ge -m$;
        \item \label{prop:Fp_c}
            if $m\le 2p-2$, then $a_{2p+m-1}(S/\calJ_{K_m,G})\ge -2p+2$.
    \end{enumerate}
\end{Lemma}

\begin{proof}
    The nullity in \ref{prop:Fp_a} follows from the fact that $\depth(S/\calJ_{K_m,G})=m+2p-1$ by \Cref{thm:fp_dim_depth}\ref{thm:fp_dim_depth_b} and the graded version of \cite[Theorem 6.2.7]{MR1613627} or \cite[Proposition 3.5.4 (b)]{MR1251956}.

    In the following, we take $v=2p-1\in V(G)$ as in \Cref{obs:f_p_v}.
    Now, consider the long exact sequence of local cohomology modules induced
    from the short sequence \eqref{eqn:SES-1} in \Cref{rem:exactseq}. From
    \Cref{thm:fan_general}\ref{thm:fan_general_a} and
    \Cref{thm:fp_dim_depth}\ref{thm:fp_dim_depth_b},
    we derive that
    \[
        0\to H_{S_+}^{2p+m-2}(S_{G_v\setminus v}/\calJ_{K_m,G_v\setminus v})
        \to H_{S_+}^{2p+m-1}(S_G/\calJ_{K_m,G}).
    \]
    Consequently,
    \[
        a_{2p+m-1}(S_G/\calJ_{K_m,G})\ge
        a_{2p+m-2}(S_{F_1^W(K_{p})}/\calJ_{K_m,F_1^W(K_{p})}).
    \]
    It remains to apply \ref{lem:fan_kn_v<m+k_b} and \ref{lem:fan_kn_v<m+k_c}
    of \Cref{lem:fan_kn_v<m+k}.
\end{proof}

Now, we are ready to give the regularity formula for the generalized binomial edge ideal of the $F_p$ graph.

\begin{Theorem}
    \label{thm:reg_Fp}
    Let $G=F_p$ with $p\ge 1$. Then, we have $\reg(S/\calJ_{K_m,G})=\min\{2p-1,m+1\}$.
\end{Theorem}

\begin{proof}
    When $p=1$ or $2$, one has $G=P_2$ or $P_4$, respectively. Whence,
    $\reg(S/\calJ_{K_m,G})=2p-1$ by \Cref{lem:reg_path}.  Meanwhile, we have
    assumed from the beginning that $m\ge 2$. So $m+1\ge 3\ge 2p-1$.

    We will now assume that $p\ge 3$. Take $v=2p-1\in V(G)$ as in
    \Cref{obs:f_p_v}. We can derive the following:
    \begin{itemize}
        \item $\reg(S_{G_v}/\calJ_{K_m,G_v})=\min\{2p-1,m\}$ by
            \Cref{thm:FkW_pure_reg};
        \item  $\reg(S_{G_v\setminus v}/\calJ_{K_m,G_v\setminus
            v})=\min\{2p-2,m\} $ by \Cref{thm:FkW_pure_reg};
        \item $\reg(S_{G\setminus v}/\calJ_{K_m,G\setminus
            v})=\min\{2p-3,m+1\}$ by inductive hypothesis.
    \end{itemize}
    Therefore, it follows from
    \Cref{depthlemma}\ref{depthlemma-c} and the short exact sequence
    \eqref{eqn:SES-1} in \Cref{rem:exactseq} that
    \[
        \reg(S/\calJ_{K_m,G})  \le \max\{\max
            \{\reg(S/\calJ_{K_m,G_v}),\reg(S'/\calJ_{K_m,G\setminus v})\},
        \reg(S'/\calJ_{K_m,G_v\setminus v})+1\}.
    \]
    It is then clear  that we have both
    \begin{align*}
        \reg(S/\calJ_{K_m,G})  &\le \max\{\max\{2p-1,2p-3\},(2p-2)+1\}= 2p-1,
        \intertext{and}
        \reg(S/\calJ_{K_m,G})  &\le \max\{\max\{m,m+1\},m+1\}= m+1.
    \end{align*}
    Thus,
    \begin{equation}
        \reg(S/\calJ_{K_m,G})  \le \min\{2p-1,m+1\}.
        \label{eqn:Fp_upper_bound}
    \end{equation}
    To show that equality holds in \eqref{eqn:Fp_upper_bound}, we have four
    cases.
    \begin{enumerate}[i]
        \item \label{thm:reg_Fp_proof_i}
            When $m\le 2p-4$, then the three regularities above are $m$, $m$ and $m+1$ respectively.  In particular,
            \[
                \max \{\reg(S_{G_v}/\calJ_{K_m,G_v}),\reg(S_{G\setminus
                v}/\calJ_{K_m,G\setminus v})\}\ne \reg(S_{G_v\setminus
            v}/\calJ_{K_m,G_v\setminus v}).
        \]
        It follows from \Cref{depthlemma}\ref{depthlemma-c} that $\reg(S/\calJ_{K_m,G}) =\max\{\max\{m,m+1\},m+1\}= m+1$, confirming the assertion in this case.

    \item \label{thm:reg_Fp_proof_ii}
        If $m=2p-3$, then by \Cref{prop:Fp}\ref{prop:Fp_c}  we have
        \begin{align*}
            \reg(S/\calJ_{K_m,G}) &\ge (2p+m-1) +
            a_{2p+m-1}(S/\calJ_{K_m,G}) \\
            &\ge 2p+m-1-2p+2 =m+1.
        \end{align*}
        Therefore, $\reg(S/\calJ_{K_m,G})=m+1$ by the inequality in
        \eqref{eqn:Fp_upper_bound}, confirming the assertion in this case.

    \item If $m=2p-2$, one can argue as in \ref{thm:reg_Fp_proof_ii} by
        \Cref{prop:Fp}\ref{prop:Fp_b}.
        Therefore, $\reg(S/\calJ_{K_m,G})=2p-1$ by the inequality in
        \eqref{eqn:Fp_upper_bound}, confirming the assertion in this case.

    \item When $m\ge 2p-1$, similarly to  \ref{thm:reg_Fp_proof_i}, we can see that $\reg(S/\calJ_{K_m,G}) =\max\{\max\{2p-1,2p-3\},(2p-2)+1\}= 2p-1$, confirming the assertion in this case.  \qedhere
\end{enumerate}
\end{proof}

\section{Depth results involving $\circ$ operations and $*$ operations}
\label{sec:depth}

Starting from this section, we will study some simple graphs obtained from the fan graphs and $F_p$ graphs by using the $\circ$ operation or the $*$ operation. The main task of this section is to provide some exact formulas for the depth of the generalized binomial edge ideals of these graphs. 
We start by recalling from \cite{MR3779601} the two special
gluing operations  mentioned earlier.

\begin{Definition}
    \label{circ_*_operations}
    For $i= 1,2$, let $G_i$ be a graph with 
    a leaf $f_i$. Furthermore, let $v_i$ be the neighbor vertex of $f_i$ in $G_i$.
    \begin{enumerate}[a]
        \item Let $G$ be the graph obtained from $G_1$ and $G_2$ by  first removing the leaves $f_1, f_2$, and then identifying the vertices $v_1$ and $v_2$.  In this case, we say that $G$ is obtained from $G_1$ and $G_2$ by the \emph{$\circ$ operation} and write $G=(G_1,f_1) \circ (G_2,f_2)$ or simply $G=G_1 \circ G_2$. If $v_1$ and $v_2$ are identified as the vertex $v$ in $G$, then we also write $G= G_1\circ_v G_2$. Unless otherwise specified, when we perform the $\circ$ operation in this way, we always implicitly assume that neither $G_1$ nor $G_2$ is the path graph $P_2$ of two vertices.

        \item Let $H$ be the graph obtained from $G_1$ and $G_2$ by identifying the vertices $f_1$ and $f_2$. In this case, we say that $H$ is obtained from $G_1$ and $G_2$ by the \emph{$*$ operation} and write $H=(G_1,f_1)*(G_2,f_2)$ or simply $H= G_1 * G_2$. If we denote the identified vertex in $H$ by $f$, then we also write $H= G_1 *_f G_2$.
    \end{enumerate}
\end{Definition}

To avoid confusion, we provide the fan graphs and the $F_p$ graphs with additional ``marked'' structures.

\begin{Definition}
    \begin{enumerate}[a]
        \item Let $G=F_k^W(K_n)$ be a $k$-fan with $k\ge 1$. When $k=1$, we say that $G$ is a \emph{marked} $1$-fan with the \emph{marked leaf} $f$ where $f$ is the unique leaf of $G$. When $k\ge 2$, we say that $G$ is a \emph{marked} $k$-fan with the \emph{marked leaves} $f_1$ and $f_2$, if $f_1$ and $f_2$ are two fixed leaves of $G$. Let $\mathbf{Fan}$ be the class of marked fans 
            and $\mathbf{Fan}^p$ be the subclass of marked pure fans.
        \item Let $\mathbf{F}$ be the class of graphs $F_p$ with $p\ge 1$. Each such $F_p$ is automatically marked with its two  unique leaves.
    \end{enumerate}
\end{Definition}

Next, we glue  the above marked graphs together  by the operations introduced in
\Cref{circ_*_operations}.

\begin{Definition}
    \label{FFan}
    Let $\mathbf{FFan}$ be the smallest class of simple graphs satisfying the
    following conditions:
    \begin{enumerate}[a]
        \item \label{def:FFan_a}
            $\mathbf{F}\cup\mathbf{Fan}\subset \mathbf{FFan}$;
        \item $\mathbf{FFan}$ is closed with respect to the $*$ operation and
            the $\circ$ operation using only the marked leaves.
    \end{enumerate}
    Moreover, if we replace $\mathbf{Fan}$ in \ref{def:FFan_a} with $\mathbf{Fan}^p$, then we obtain the subclass $\mathbf{FFan}^p$.
    Note that if $G_1,G_2,G_3\in \mathbf{F}\cup\mathbf{Fan}$, then $(G_1
    \circledast_1 G_2)\circledast_2 G_3=G_1 \circledast_1 (G_2 \circledast_2
    G_3)$ whenever each $\circledast_i$ is either a $*$ operation or a $\circ$
    operation using only the marked leaves.  Therefore, for each $G\in
    \mathbf{FFan}$, we can find $G_1,\dots,G_t$ such that
    \begin{enumerate}[a]
        \item each $G_i$ belongs to $\mathbf{F}\cup \mathbf{Fan}$;
        \item $G=G_1 \circledast_1 \cdots \circledast_{t-1} G_t$ where each
            $\circledast_i$ is either a $*$ operation or a $\circ$ operation
            using only marked leaves.
    \end{enumerate}
    Notice that $F_1=P_2$ and $F_2=P_4=F_1 * F_1 * F_1$. Thus, the above decomposition of $G$ is not unique in general. Without loss of generality, we can assume that $t$ is the minimum in which this happens. Then, we will say that $G$ has $t$ \emph{$\mathbf{FFan}$-components} and write $t=\ffc(G)$.

\end{Definition}

If $G$ is a connected bipartite graph, then the binomial edge ideal $J_G$ is Cohen--Macaulay if and only if $G\in \mathbf{FFan}$ by \cite[Theorem 6.1]{MR3779601}. It is clear that in this case, every $\mathbf{FFan}$ component of $G$ will necessarily belong to $\mathbf{F}$.

\begin{Remark}
    If we do not specify the marked leaves, then the $*$ operation and the $\circ$ operation will not be associative. For example, let $G_2$ be a fan graph with at least $3$ leaves in different branches. Let $G_1,G_3,G_4$ be all $\mathbf{F}$ graphs. Suppose that $G=((G_1*G_2)*G_3)*G_4$, where we use  all the leaves in $G_2$. Obviously, $G\ne (G_1*G_2)*(G_3*G_4)$, since in $G_3*G_4$ we cannot use leaves from $G_2$.
\end{Remark}

In the following, we first present  three basic scenarios when using
the $\circ$ operation.

\begin{Observation}
    \label{obs:circ}
    For $i=1,2$, let $G_i=F_{n_i}$ with $n_i\ge 2$. Meanwhile, for $i=3,4$,
    let $G_i= F_{k_i}^{W_i}(K_{n_i})$ be a $k_i$-fan graph on $W_i=W_{i1}
    \sqcup \cdots \sqcup W_{ik_i}$ with $k_i\ge 1$ and $n_i\ge 3$. 
    Here, we consider $G=G_{i_1}\circ_v G_{i_2}$. Let $H$ be the complete graph on the vertex set $N_G[v]$.
    \begin{enumerate}[a]
        \item   \label{obs:circ_a}If $G= G_1\circ_v G_2$,  then
            \begin{itemize}
                \item $G_v=F_2^W(H)$ is a $2$-pure fan graph of $H$ on
                    $W\coloneqq N_G(v)$,
                \item $G_v\setminus v=F_2^W(H\setminus v)$ is a $2$-pure fan
                    of $H\setminus v$ on $W$, and
                \item $G\setminus v=F_{n_1-1}\sqcup F_{n_2-1}$.
            \end{itemize}
        \item \label{obs:circ_b}
            Suppose that
            $G=(G_3,f_3)\circ_v (G_4,f_4)$ and $f_i\in W_{i1}$ for $i=3,4$.
            Then, it is easy to observe that
            \begin{itemize}
                \item $G_v$ is a $(k_3+k_4-2)$-fan graph of $H$ on $W
                    \coloneqq (W_3 \setminus W_{31}) \sqcup (W_4 \setminus
                    W_{41})$,
                \item $G_v\setminus v$ is a $(k_3+k_4-2)$-fan graph of
                    $H\setminus v$ on $W$, and
                \item $G\setminus v$ is the disjoint union of two fan graphs.
            \end{itemize}
        \item  \label{obs:circ_c}
            Suppose that $G=(G_1,f_1)\circ_v (G_4,f_4)$ and $f_4\in W_{41}$.
            Then, it is easy to observe that
            \begin{itemize}
                \item $G_v=F_{k_4}^{W}(H)$ is a $k_4$-fan graph of $H$ on $W
                    \coloneqq N_{G_1\setminus f_1}(v) \sqcup (W_4\setminus
                    W_{41})$,
                \item $G_v\setminus v=F_{k_4}^{W}(H\setminus v)$ is a
                    $k_4$-fan of $H\setminus v$ on $W$, and
                \item $G\setminus v$ is the disjoint union of $F_{n_1-1}$ and
                    a fan graph.
            \end{itemize}
    \end{enumerate}
    Furthermore, in each case we have  $|V(G_v)|= |V(G)|= |V(G_1)|+ |V(G_2)|-3$, and $|V(G_v\setminus v)|=|V(G\setminus v)|= |V(G)|-1$.
\end{Observation}

\begin{Lemma}
    \label{lem:depth_circ_trick}
    Let $M$ be a positive integer. Suppose that for every $H$ in
    $\mathbf{FFan}$ with $\ffc(H)\le M$, we always have
    $\depth(S_H/\calJ_{K_m,H}) =m+|V(H)|-1$. Now, let $G_1$ and $G_2$ be two
    members of $\mathbf{FFan}$ such that $\ffc(G_1)+\ffc(G_2)\le M+1$. Then,
    for $G=G_1\circ G_2$, we also have $\depth(S_G/\calJ_{K_m,G})=m+|V(G)|-1$.
\end{Lemma}
\begin{proof}
    We have   $\ffc(G)\le \ffc(G_1)+\ffc(G_2)$.
    Suppose that $G= G_1 \circ_v G_2$.
    It follows from \Cref{obs:circ} that both $G_v$ and $G_v\setminus v$
    belong to $\mathbf{FFan}$, such that 
    $\ffc(G_v)\le \ffc(G)-1$ and 
    $\ffc(G_v\setminus v) \le \ffc(G)-1$.
    Meanwhile, $G\setminus v$ is the disjoint union of $G_1'$
    and $G_2'$, where $G_i'\coloneqq G_i\setminus \{f_i,v_i\}\in
    \mathbf{FFan}$ such that $\ffc(G_i')\le \ffc(G)$ for $i=1,2$. Thus, by assumption, we
    have
    \begin{align*}
        \depth(S_{G_v}/\calJ_{K_m,G_v}) & =m+|V(G)|-1,\\
        \depth(S_{G_v\setminus v}/\calJ_{K_m,G_v\setminus v}) & =m+|V(G)|-2,
        \\
        \intertext{and}
        \depth(S_{G\setminus v}/\calJ_{K_m,G\setminus v}) & =m+|V(G)|-1+(m-2).
    \end{align*}
    Applying
    \Cref{depthlemma}\ref{depthlemma-a} to the short exact sequence \eqref{eqn:SES-1} in \Cref{rem:exactseq}, we  obtain the desired formula.
\end{proof}

Next, we consider the application of $*$ operations. As a preparation, we need
the terminology of clique sum.

\begin{Definition}
    [{\cite[Definition 3.10]{JKS}}]
    Let $G_1$ and $G_2$ be two subgraphs of a graph $G$. If $G_1 \cap G_2 =
    K_\ell$, the complete graph on $\ell$ vertices with $G_1 \ne K_\ell$ and
    $G_2 \ne K_{\ell}$, then $G$ is called the \emph{clique sum} of $G_1$ and
    $G_2$ along the complete graph $K_{\ell}$, denoted by $G_1 \cup_{K_\ell}
    G_2$. If $\ell = 1$, the clique sum of $G_1$ and $G_2$ along a vertex $v$
    is denoted by $G_1 \cup_v G_2$ for short.
\end{Definition}

Obviously, gluing via the $*$ operation is a special clique sum in the case of $\ell=1$; the latter does not insist on involving leaves. Clique sum in the following scenario often occurs when we are dealing with gluing via the $*$ operations.

\begin{Observation}
    \label{obs:clique_sum}
    Let $G=G_1\cup_u G_2$ be the clique sum of two graphs $G_1$ and $G_2$.
    Suppose that $G_{1}= F_{k}^{W}(K_{n})$ is a $k$-fan graph of $K_{n}$ on $W$
    and $u\in V(K_n)\setminus W$. Suppose also that $u$ is a leaf
    of $G_2$ with the unique neighbor point $v$. Let $H$ be the complete
    graph on vertex set $N_G[v]$ and 
    $U=N_G(v)$.
    \begin{enumerate}[a]
        \item If $G_2=F_1$, then $G=F_{k+1}^{W\cup\{u\}}(K_n)$.
        \item Suppose that $G_2=F_p$ for some $p\ge 2$. Then $G_v$ is obtained
            from the fan graphs $F_{k+1}^{W\cup \{u\}}(K_{n})$ and
            $F_{2}^{U}(H)$ by the $\circ$ operation. Similarly, $G_v\setminus
            v= F_{k+1}^{W\cup \{u\}}(K_{n}) \circ F_{2}^{U}(H\setminus v)$.
            Furthermore, $G\setminus v$ is the disjoint union of $G_1$ and
            $F_{p-1}=G_2\setminus\{u,v\}$.
        \item Suppose that $G_2=F_{k'}^{W'}(K_{n'})$ is a fan graph with
            $k'\ge 1$ with $W'=W_1'\sqcup \cdots \sqcup W_{k'}'$ and $v\in
            W_1'$. Then $G_v = F_{k+1}^{W\cup \{u\}}(K_{n}) \circ
            F_{k'}^{(W'\setminus W_1') \cup \{u\}}(H)$ and $G_v\setminus v =
            F_{k+1}^{W\cup \{u\}}(K_{n}) \circ F_{k'}^{(W'\setminus W_1')\cup
            \{u\}}(H\setminus v)$. Furthermore, $G\setminus v$ is the disjoint
            union of the fan graphs $G_1$ and $G_2\setminus\{u,v\}$.
    \end{enumerate}
    Note that in each case, $|V(G_v)|=|V(G)|$ and $|V(G_v\setminus v)|=|V(G\setminus v)|=|V(G)|-1$.
\end{Observation}

\begin{Lemma}
    \label{lem:clique_sum}
    Let $M$ be a positive integer. Suppose that for every $H$ in
    $\mathbf{FFan}$ with $\ffc(H)\le M$, we always have
    $\depth(S_H/\calJ_{K_m,H})=m+|V(H)|-1$.
    Let $G=G_1\cup_u G_2$ be the
    clique sum of two graphs $G_1$ and $G_2$.
    Suppose  further that $G_{1}=F_{k}^{W}(K_{n})$ is a $k$-fan graph of
    $K_{n}$ on $W$ and $u\in V(K_n)\setminus W$. At the same time, suppose
    that $G_2\in \mathbf{FFan}$ with $\ffc(G_2)\le M$ and $u$ is a marked
    leaf of $G_2$.  Then, $\depth(S_G/\calJ_{K_m,G})=m+|V(G)|-1$.
\end{Lemma}

\begin{proof}
    With \Cref{obs:clique_sum}, we can argue as in the proof of
    \Cref{lem:depth_circ_trick}.
\end{proof}

 Finally, we are ready to present the promised depth result.

\begin{Theorem}
    \label{thm:full_depth_star_circ}
    For each $G\in \mathbf{FFan}$, we have
    $\depth(S/\calJ_{K_m,G})=m+|V(G)|-1$.
\end{Theorem}

\begin{proof}
    We prove by induction on $\ffc(G)$. The base case is when $\ffc(G)=1$,
    which follows from \Cref{thm:fan_general}\ref{thm:fan_general_a} and
    \Cref{thm:fp_dim_depth}\ref{thm:fp_dim_depth_b}. In the following, we assume
    that $\ffc(G)\ge 2$. Suppose that $G=(G_1,f_1)\circledast (G_2,f_2)$ with
    $\ffc(G_1)=\ffc(G)-1$ and $G_2\in \mathbf{F}\cup \mathbf{Fan}$.
    Obviously, we have two cases to check.
    \begin{enumerate}[a]
        \item Suppose that $\circledast=\circ$. In this case, we can apply
            \Cref{lem:depth_circ_trick}.
        \item Suppose that $\circledast=*$. In this case, let $v$ be the
            unique neighbor point of $f_2$ in $G_2$. Suppose that $H$ is the
            complete graph on $N_{G_2}[v]$. Then $G_v$ is the clique sum $G_1
            \cup_{f_2} G_2'$ where $G_2'$ is a $k'$-fan graph of $H$ on some
            subset $W\subset V(G_2)$ for some positive integer $k'$.
            Obviously, $G_v\setminus v$ is the clique sum $G_1 \cup_{f_2}
            (G_2'\setminus v)$ where $G_2'\setminus v$ is a $k'$-fan graph of
            $H\setminus v$ on the same $W$. Thus, by \Cref{lem:clique_sum}, we
            have $\depth(S/\calJ_{K_m,G_v}) =m+|V(G)|-1$ and
            $\depth(S'/\calJ_{K_m,G_v\setminus v})= m+|V(G)|
            -2$. Meanwhile, $G\setminus v$ is the disjoint union of $G_1$ and
            $G_2\setminus \{f_2,v\}$. Since $G_2\setminus \{f_2,v\}$
            still belongs to $\mathbf{F}\cup \mathbf{Fan}$, we have
            $\depth(S_{G\setminus v}/\calJ_{K_m,G\setminus v})=m+|V(G_1)|-1+m
            +|V(G_2)|-2-1=m+|V(G)|-1+(m-2)$ by induction.
            Applying \Cref{depthlemma}\ref{depthlemma-a} to the exact sequence (\ref{eqn:SES-1})  in \Cref{rem:exactseq}, we obtain
            the desired formula.  \qedhere
    \end{enumerate}
\end{proof}

\begin{Corollary}
    \label{cor:compare_a_inv}
    Suppose that $G$ is a graph in $\mathbf{FFan}$ and $v$ is an internal vertex in $G$. Assume that both $G_v$ and $G_v\setminus v$ belong to $\mathbf{FFan}$. Suppose further that $G\setminus v$ is the disjoint union of $G_1$ and $G_2$, both of which also belong to $\mathbf{FFan}$. Then, we have an exact sequence of local cohomology modules:
    \[
        0\to H_{S_+}^{m+|V(G)|-2}(S/\calJ_{K_m,G_v\setminus v}) \to H_{S_+}^{m+|V(G)|-1}(S/\calJ_{K_m,G}).
    \]
\end{Corollary}
\begin{proof}
    Since these graphs all belong to $\mathbf{FFan}$ and  $|V(G)|=|V(G\setminus v)|+1=|V(G_1)|+|V(G_2)|+1$, we have
    \begin{align*}
        \depth(S/\calJ_{K_m,G})               & =
        \depth(S/\calJ_{K_m,G_v})=m+|V(G)|-1,                               \\
        \depth(S/\calJ_{K_m, G_v\setminus v}) & = m+|V(G)|-2,               \\
        \intertext{and}
        \depth(S/\calJ_{K_m, G\setminus v})   & =
        \depth(S/\calJ_{K_m, G_1})+ \depth(S/\calJ_{K_m, G_2})              \\
        & = m+|V(G_1)|-1+m+|V(G_2)|-1 \\
        & =2m+|V(G)|-3>m+|V(G)|-2.
    \end{align*}
    Thus, the desired exact sequence of local cohomology modules follows from the short exact sequence (\ref{eqn:SES-1}) in \Cref{rem:exactseq}.
\end{proof}

\section{Regularity results involving $\circ$ operations and $*$ operations}
\label{sec:regularity}

In this section, we will compute the regularity of the generalized binomial edge ideals of graphs in $\mathbf{FFan}$.
To take advantage of the formula in \Cref{thm:FkW_pure_reg}, we focus on the graphs in $\mathbf{FFan}^p$.

\subsection{Glue via the $\circ$ operation}
First of all, we compute the regularity of the generalized binomial edge ideal of a graph in $\mathbf{FFan}^p$, which is obtained from $\mathbf{F} \cup \mathbf{Fan}^p$ by a single $\circ$ operation.

\begin{Remark}
    Suppose that $G=F_k^W(K_n)$ is a $k$-pure fan of $K_n$ on $W$. If $n=2$, since we always have $k\le n$, it is easy to see that $G$ can also be viewed as a path graph, or a $1$-pure fan of $K_3$ on some $W'$. 
    Because of this simple observation, we are only interested in the case when $n\ge 3$ in this subsection. The case when the gluing involves a path graph will be  treated in the next subsection.
\end{Remark}

\begin{Definition}
    To simplify the notations and arguments below, when we say that a graph $G$ is a \emph{$k$-pure pseudo fan graph}, we mean that it is either a $k$-pure fan graph $F_k^W(K_n)$, or the $F_p$ graph with $k=2$.
\end{Definition}

The following result is a generalization of \cite[Propositions 4.3 and 4.4]{MR3991052} where $m=2$.

\begin{Proposition}
    \label{prop:reg_Fan_Fan_circ}
    Let $G=(G_1,f_1) \circ (G_2,f_2)$, where $G_i$ is a $k_i$-pure pseudo fan graph with $|V(G_i)|-k_i\ge 3$ and $k_i\ge 1$ for each $i$. Suppose also that $v_i$ is the neighbor point of $f_i$ in $G_i$ and $G_i\setminus \{v_i,f_i\}$ is a $q_i$-pure pseudo fan graph for each $i$. In the end, assume that $v_1$ and $v_2$ are identified as the vertex $v$ in $G$. Then, we have
    \begin{align}
        \reg(S/\calJ_{K_m,G}) & \le \max\Big\{
            \min\{|V(G_1)|-3,m+q_1-1\}+\min\{|V(G_2)|-3,m+q_2-1\}, \notag \\
            & \quad \min\{|V(G)|-1,m+k_1+k_2-3\},\min\{|V(G)|-1,m+k_1+k_2-2\}
            \label{eqn:fan_fan_circ_reg_upper_bound}
        \Big\}.
    \end{align}
    Furthermore, if $q_i=k_i$ for each $i$, then we have an equality in \eqref{eqn:fan_fan_circ_reg_upper_bound}.
\end{Proposition}

\begin{proof}
    It is clear that $q_i\in \{k_i-1,k_i\}$ for each $i$. Note that the graph $G\setminus v$ is the disjoint union of a $q_1$-pure pseudo fan graph and a $q_2$-pure pseudo fan graph. Furthermore, $|V(G)|=|V(G_1)|+|V(G_2)|-3$. Thus, it follows from \Cref{thm:FkW_pure_reg}, \Cref{thm:reg_Fp} and \Cref{obs:circ} that
    \begin{align*}
        \reg(S/\calJ_{K_m,G_v}) & =\min\{|V(G)|-1,m+k_1+k_2-2-1\}, \\
        \reg(S/\calJ_{K_m,G_v\setminus v}) & =\min\{|V(G)|-2,m+k_1+k_2-2-1\},\\
        \intertext{and}
        \reg(S/\calJ_{K_m,G\setminus v}) & =\min\{|V(G_1)|-2-1,m+q_1-1\}\\ 
        &\qquad +\min\{|V(G_2)|-2-1,m+q_2-1\}.
    \end{align*}
    Applying \Cref{depthlemma}\ref{depthlemma-c} to the short exact sequence \eqref{eqn:SES-1} in \Cref{rem:exactseq}, we have
    \begin{equation}
        \reg(S/\calJ_{K_m,G})\le \max\{\reg(S/\calJ_{K_m,G\setminus
        v}),\reg(S/\calJ_{K_m,G_v}),\reg(S/\calJ_{K_m,G_v\setminus v})+1\},
        \label{eqn:fan_fan_circ_reg_upper_bound_coarse}
    \end{equation}
    where the RHS of \eqref{eqn:fan_fan_circ_reg_upper_bound_coarse} is exactly the RHS of \eqref{eqn:fan_fan_circ_reg_upper_bound}. This proves the first part of \Cref{prop:reg_Fan_Fan_circ}.

    In the rest of the proof, we assume that $q_i=k_i$ for each $i$. There are two cases to consider for the expected equality.
    \begin{enumerate}[i]
        \item \label{prop:reg_Fan_Fan_circ_i}
            If $m+k_1+k_2-|V(G)|\notin \{0,1\}$, then we claim that
            \[
                \max\{\reg(S/\calJ_{K_m,G_v}),\reg(S/\calJ_{K_m,G\setminus
                v})\}>
                \reg(S/\calJ_{K_m,G_v\setminus v}).
            \]
            With this, we have an equality in \eqref{eqn:fan_fan_circ_reg_upper_bound_coarse} by \Cref{depthlemma}\ref{depthlemma-c}. To confirm the claim, we check with two subcases by elementary arguments.
            \begin{enumerate}[a]
                \item \label{prop:reg_Fan_Fan_circ_i_a}
                    Suppose that $m+k_1+k_2-|V(G)|\le -1$. In this case, we have $\reg(S/\calJ_{K_m,G_v\setminus v})=m+k_1+k_2-3$. Next, we show that $\reg(S/\calJ_{K_m,G\setminus v})>m+k_1+k_2-3$ by checking with the four possibilities when using the above formula for $\reg(S/\calJ_{K_m,G_v\setminus v})$. 
                    \begin{itemize}
                        \item Since $|V(G)|=|V(G_1)|+|V(G_2)|-3$, we have $|V(G_1)|-3+|V(G_2)|-3>m+k_1+k_2-3$ by our assumption. 
                        \item Meanwhile, $m+k_1-1+m+k_2-1>m+k_1+k_2-3$, because $m\ge 2$. 
                        \item At the same time, for $\{i,j\}=\{1,2\}$, we have $|V(G_i)|-2-1+m+k_j-1>m+k_1+k_2-3$, since $|V(G_i)|-k_i \ge 3$.
                    \end{itemize}
                    In summary, we have $\reg(S/\calJ_{K_m,G\setminus v})>m+k_1+k_2-3$. As a result, the claim holds in this subcase.
                \item Suppose that $m+k_1+k_2-|V(G)|\ge 2$. In this case, $\reg(S/\calJ_{K_m,G_v})=|V(G)|-1> \reg(S/\calJ_{K_m,G_v\setminus v}) =|V(G)|-2$.  As a result, the claim still holds in this subcase.
            \end{enumerate}
        \item  \label{prop:reg_Fan_Fan_circ_ii}
            If instead  $m+k_1+k_2-|V(G)|\in \{0,1\}$, then we claim that we still have an equality in \eqref{eqn:fan_fan_circ_reg_upper_bound_coarse}. Note that in this case, $\reg(S/\calJ_{K_m,G_v\setminus v})=\reg(S/\calJ_{K_m,G_v})=m+k_1+k_2-3$, and $\reg(S/\calJ_{K_m, G\setminus v})\le |V(G_1)|-3+|V(G_2)|-3\le m+k_1+k_2-3$. Thus, we get
            \begin{equation}
                \reg(S/\calJ_{K_m,G})\le m+k_1+k_2-2 = \text{RHS of \eqref{eqn:fan_fan_circ_reg_upper_bound_coarse}}.
                \label{eqn:fan_fan_circ_reg_upper_bound_sharp}
            \end{equation}
            On the other hand, we obtain $|V(G_v\setminus v)|=|V(G)|-1\ge m+(k_1+k_2-2)$ by the condition $m+k_1+k_2-|V(G)|\in \{0,1\}$. It follows from \Cref{lem:fan_kn_v<m+k}\ref{lem:fan_kn_v<m+k_c} and \Cref{cor:compare_a_inv} that 
            \begin{align*}
                a_{m+|V(G)|-1}(S_G/\calJ_{K_m,G}) &\ge
                a_{m+|V(G)|-2}(S_{G_v\setminus
                v}/\calJ_{K_m,G_v\setminus v})\\
                &\ge -|V(G)|+1+(k_1+k_2-2).
            \end{align*}
            Therefore,
            \begin{align*}
                \reg(S_G/\calJ_{K_m,G}) & \ge
                a_{m+|V(G)|-1}(S_G/\calJ_{K_m,G})+(m+|V(G)|-1) \\
                & \ge -|V(G)|+k_1+k_2-1+m+|V(G)|-1 \\
                & =m+k_1+k_2-2.
            \end{align*}
            Combining this with \eqref{eqn:fan_fan_circ_reg_upper_bound_sharp}, we have $\reg(S_G/\calJ_{K_m,G})=  m+k_1+k_2-2$, which establishes the claim in this case. 
            \qedhere
    \end{enumerate}
\end{proof}

\begin{Remark}
    \label{rmk:reg_Fan_Fan_circ}
    Under the assumptions of \Cref{prop:reg_Fan_Fan_circ}, suppose additionally that $n_i\ge 3$.
    \begin{enumerate}[a]
        \item If $G_i=F_{k_i}^{W_i}(K_{n_i})$ is a $k_i$-pure fan graph, then $|V(G_i)|-k_i\ge n_i\ge 3$. Moreover, if $k_i=q_i$, then the deletion graph $G_i\setminus \{v_i,f_i\}=F_{k_i}^{W_i\setminus \{v_i\}}(K_{n_i-1})$ is still a $k_i$-pure fan graph. Thus, $|W_i|-1\ge k_i$. From this, we can conclude that $|V(G_i)|-k_i \ge n_i+|W_i|-k_i \ge n_i+1\ge 4$. As a result, $k_i\le |V(G_i)|- 4$.

        \item On the other hand, if $G_i=F_{n_i}$, it is also clear that $k_i=q_i=2$ and $|V(G_i)|-k_i=|V(G_i)|-2\ge 4$.
    \end{enumerate}
\end{Remark}

We can draw two conclusions from \Cref{prop:reg_Fan_Fan_circ} and its proof. The first one is direct.

\begin{Corollary}
    \label{cor:reg_Fan_fan_upper_bound}
    Under the assumptions of \Cref{prop:reg_Fan_Fan_circ}, we have
    \[
        \reg(S/\calJ_{K_m,G})\le (m+q_1-1)+(m+q_2-1).
    \]
    Moreover, if $m\le \min\{|V(G_1)|-q_1-2,|V(G_2)|-q_2-2\}$, then equality holds here.
\end{Corollary}

\begin{proof}
    Notice that
    \begin{align*}
        \text{RHS of \eqref{eqn:fan_fan_circ_reg_upper_bound}} & \le
        \max\{(m+q_1-1)+(m+q_2-1),m+k_1+k_2-3,m+k_1+k_2-2\} \\
        & =(m+q_1-1)+(m+q_2-1),
    \end{align*}
    since $k_1\le q_1+1$, $k_2\le q_2+1$ and $m\ge 2$.

    Moreover, if $m\le \min\{|V(G_1)|-q_1-2,|V(G_2)|-q_2-2\}$, then
    \begin{align*}
        m+k_1+k_2-|V(G)| & \le m+q_1+1+q_2+1-|V(G_1)|-|V(G_2)|+3+(m-2) \\
        & =(m+q_1+2-|V(G_1)|)+(m+q_2+2-|V(G_2)|)-1    \\
        & \le -1.
    \end{align*}
    This leads to the subcase \ref{prop:reg_Fan_Fan_circ_i}\ref{prop:reg_Fan_Fan_circ_i_a} in the proof of \Cref{prop:reg_Fan_Fan_circ}, where we have
    \[
        \max\{\reg(S/\calJ_{K_m,G_v}),\reg(S/\calJ_{K_m,G\setminus v})\} > \reg(S/\calJ_{K_m,G_v\setminus v})
    \]
    Hence, we have an equality in \eqref{eqn:fan_fan_circ_reg_upper_bound}, whose RHS is exactly $(m+q_1-1)+(m+q_2-1)$.
\end{proof}

The second one is more technical.

\begin{Corollary}
    \label{rem:reg_fan_fan}
    Under the assumptions of \Cref{prop:reg_Fan_Fan_circ}, suppose additionally that $q_i=k_i$ for each $i$.
    Furthermore, assume that $|V(G_2)|-k_2\ge 4$ and $m\le |V(G_1)| -k_1$.
    Then, we have
    \begin{align*}
        \reg(S/\calJ_{K_m,G_v}) &<
        \reg(S/\calJ_{K_m,G})=\reg(S/\calJ_{K_m,G\setminus v})\\
        &=\reg(S/\calJ_{K_m,G_1\setminus
        \{v_1,f_1\}})+\reg(S/\calJ_{K_m,G_2\setminus \{v_2,f_2\}}).
    \end{align*}
\end{Corollary}

\begin{proof}
    From the assumptions we see that $m+k_1+k_2\le |V(G_1)|+(|V(G_2)|-4)=|V(G)|-1$. Whence,  we have
    \begin{align*}
        \reg(S/\calJ_{K_m,G_v})
        =m+k_1+k_2-3=
        \reg(S/\calJ_{K_m,G_v\setminus v}).
    \end{align*}
    Furthermore,  we obtain from the subcase
    \ref{prop:reg_Fan_Fan_circ_i}\ref{prop:reg_Fan_Fan_circ_i_a} in the proof
    of  \Cref{prop:reg_Fan_Fan_circ} that
    \[
        \reg(S/\calJ_{K_m,G\setminus v})= \reg(S/\calJ_{K_m,G_1\setminus
        \{v,f_1\}})+\reg(S/\calJ_{K_m,G_2\setminus \{v,f_2\}})
        \ge  m+k_1+k_2-2.
    \]
    Consequently, the equality in \eqref{eqn:fan_fan_circ_reg_upper_bound}
    implies that
    \[
        \reg(S/\calJ_{K_m,G_v})<\reg(S/\calJ_{K_m,G})=\reg(S/\calJ_{K_m,
        G\setminus v}),
    \]
    as claimed.
\end{proof}

\begin{Remark}
    If $G_1=F_{n_1}$ and $G_2=F_{n_2}$, then
    \Cref{cor:reg_Fan_fan_upper_bound} says that when $m\le
    2\min\{n_1,n_2\}-4$, one has $\reg(S/\calJ_{K_m,G})= 2m+2$. This condition
    is sharp. For instance, we can take $n_1=m=3$ and $n_2=4$. Obviously,
    $2n_1-4<m<2n_2-4$.  At the same time, $\reg(S/\calJ_{K_m},G)= 7<2m+2$ by
    \Cref{prop:reg_Fan_Fan_circ}.
\end{Remark}

In the following, we conclude this subsection by considering the relation between the regularity of $G=G_1\circ G_2$ and the regularities of the induced subgraphs of $G_1$ and $G_2$.

\begin{Theorem}
    \label{thm:circ_lower_upper_bound}
    Let $G=(G_1,f_1) \circ (G_2,f_2)$, where $G_i$ is a $k_i$-pure pseudo fan graph with $|V(G_i)|-k_i\ge 3$ and $k_i\ge 1$ for each $i$. Suppose  further that $v_i$ is the neighbor point of $f_i$ in $G_i$ for each $i$, such that $v_1$ and $v_2$ are identified as the vertex $v$ in $G$. Then, we have
    \begin{equation}
        \sum\limits_{i=1}^2\reg(S/\calJ_{K_m,G_i\setminus \{f_i,v_i\}})\le
        \reg(S/\calJ_{K_m,G})\le\sum\limits_{i=1}^2\reg(S/\calJ_{K_m,G_i}).
        \label{eqn:circ_lower_upper_bound}
    \end{equation}
\end{Theorem}
\begin{proof}
    The first inequality in \eqref{eqn:circ_lower_upper_bound} is obvious by \Cref{cor:induced_graph}, since $(G_1\setminus \{f_1,v_1\}) \sqcup (G_2\setminus \{f_2,v_2\}) = G \setminus v$ is an induced subgraph of $G$.

    Next, we consider the second inequality in \eqref{eqn:circ_lower_upper_bound}. Suppose that $G_i\setminus \{v_i, f_i\}$ is a $q_i$-pure pseudo fan graph for each $i$.
    By \Cref{prop:reg_Fan_Fan_circ}, we have
    \begin{align*}
        \reg(S/\calJ_{K_m,G}) & \le \max\Big\{
            \min\{|V(G_1)|-3,m+q_1-1\}+\min\{|V(G_2)|-3,m+q_2-1\}, \\
            & \qquad \min\{|V(G)|-1,m+k_1+k_2-3\},
        \min\{|V(G)|-1,m+k_1+k_2-2\}\Big\}.
    \end{align*}
    For simplicity,  let $A\coloneqq \min\{|V(G_t)|-k_t:t\in [2]\}$ and
    $B\coloneqq \max\{|V(G_t)|-k_t:t\in [2]\}$. Suppose that
    $\{j_1,j_2\}=\{1,2\}$ such that $A=|V(G_{j_1})|-k_{j_1}$ and
    $B=|V(G_{j_2})|-k_{j_2}$. By \Cref{thm:FkW_pure_reg}, \Cref{thm:reg_Fp}
    and the fact that $|V(G)|=|V(G_1)|+|V(G_2)|-3$, we can get
    \begin{align*}
        \sum\limits_{i=1}^2\reg(S/\calJ_{K_m,G_i}) &
        =\min\{|V(G_1)|-1,m+k_1-1\}+\min\{|V(G_2)|-1,m+k_2-1\}  \\
        & =\begin{cases}
            2m+k_1+k_2-2,             & \text{if $m\le A$,}        \\
            m+k_{j_2}+|V(G_{j_1})|-2, & \text{if $A+1\le m\le B$,} \\
            |V(G)|+1,                 & \text{if $ m\ge B+1$.}
        \end{cases}
    \end{align*}
    So we prove the second inequality in \eqref{eqn:circ_lower_upper_bound} with respect to these three cases.

    Firstly, note that $q_t\le k_t$ for each $t$. Thus $\min\{|V(G_1)|-3,m+q_1-1\}+\min\{|V(G_2)|-3,m+q_2-1\}\le (m+q_1-1)+(m+q_2-1)\le 2m+k_1+k_2-2$. Meanwhile, $\max\{\min\{|V(G)|-1,m+k_1+k_2-3\},\min\{|V(G)|-1,m+k_1+k_2-2\}\}\le \max\{m+k_1+k_2-3,m+k_1+k_2-2\}=m+k_1+k_2-2<2m+k_1+k_2-2$. This implies that $\reg(S/\calJ_{K_m,G})\le 2m+k_1+k_2-2$ for all $m$. In particular, when $m\le A$, the second inequality in \eqref{eqn:circ_lower_upper_bound} holds.

    Secondly, note that $\min\{|V(G_1)|-3, m+q_1-1\}+\min\{|V(G_2)|-3,m+q_2-1\}\le m+q_{j_2}-1+|V(G_{j_1})|-3 <m+k_{j_2}+|V(G_{j_1})|-2$. Furthermore, $\max\{\min\{ |V(G)|-1,m+k_1+k_2-3\},\min\{|V(G)|-1,m+k_1+k_2-2\}\}\le \max\{m+k_1+k_2-3, m+k_1+k_2-2\}=m+k_1+k_2-2=m+k_{j_2}+k_{j_1}-2< m+k_{j_2}+|V(G_{j_1})|-2$.  This implies that $\reg(S/\calJ_{K_m,G})\le m+k_{j_2}+|V(G_{j_1})|-2$ for all $m$. In particular, if $A+1\le m\le B$, the second inequality in \eqref{eqn:circ_lower_upper_bound} holds.

    Finally, it follows from \Cref{lem:reg_min_equal} that $\reg(S/\calJ_{K_m,G}) \le |V(G)|-1< |V(G)|+1$. In particular, when $m\ge B+1$, the second inequality in \eqref{eqn:circ_lower_upper_bound} holds.
\end{proof}

\subsection{Glue via the $*$ operation}
In this subsection, we study the regularity of the generalized binomial edge ideal of a graph in $\mathbf{FFan}^p$, which is obtained from $\mathbf{F} \cup \mathbf{Fan}^p$ by a single $*$ operation. This case is more complicated than that of the previous subsection.

In contrast to \cite[Theorem 3.1]{MR3941158} for the case of the binomial edge ideal with $m=2$, we don't have the additive formula
\[
    \reg(S/\calJ_{K_m,G_1*G_2})=\reg(S/\calJ_{K_m,G_1})+\reg(S/\calJ_{K_m,G_2})
\]
in general. In fact, we will show in \Cref{exam:333_star} that
\[
    \reg(S/\calJ_{K_4,F_3*F_3})\le 8 < 2 \reg(S/\calJ_{K_4,F_3})=10.
\]
Therefore, we turn to expect a subadditive formula as
\[
    \reg(S/\calJ_{K_m,G_1*G_2})\le \reg(S/\calJ_{K_m,G_1})+\reg(S/\calJ_{K_m,G_2}),
\]
at least for the graphs considered in this paper. For this we need some preparations.

\begin{Setting}
    \label{setting:fan_path_clique_sum}
    Let $G=G_1\cup_v G_2$ be the clique sum of two graphs $G_1$ and $G_2$.
    Suppose that $G_{1}=F_{k}^{W}(K_{n})$ is
    a $k$-pure fan graph of $K_{n}$ on $W$ and $v\in V(K_{n})\setminus W$.
    Meanwhile, suppose that $G_2=P_t$ is a path with $t$ vertices for $t\ge
    2$ and $v$ is a leaf of $G_2$.
\end{Setting}

\begin{Observation}
    \label{obs:fan_path_clique_sum}
    With \Cref{setting:fan_path_clique_sum}, let $u$ be the unique neighbor
    point of $v$ in $P_t$. It is clear that $G_v$ has the form
    $F_k^{W}(K_{n+1})\cup_u P_{t-1} $ and $G_v\setminus v$ has the form $
    F_k^{W}(K_{n})\cup_u P_{t-1}$. At the same time, $G\setminus v$ is the
    disjoint union of  $F_k^W(K_{n-1})$ with $P_{t-1}$. Furthermore, $G=
    F_{k+1}^{W\cup\{v\}}(K_{n})*P_{t-1} =F_{k+1}^{W\cup\{v\}}(K_{n})*\underbrace{ F_1*F_1*
    \cdots*F_1}_{t-2}\in \mathbf{FFan}^p$.
\end{Observation}

\begin{Proposition}
    \label{prop:fan_path_clique_sum}
    Under the assumptions of \Cref{setting:fan_path_clique_sum}, we have
    \begin{align}
        \reg(S/\calJ_{K_m,G}) & \le
        \reg(S/\calJ_{K_m,G_1})+\reg(S/\calJ_{K_m,G_2}).
        \label{eqn:path_fan_reg_bound}
    \end{align}
    Furthermore, if $m\ge k+1$, then we have an equality here.
\end{Proposition}

Note that the RHS of \eqref{eqn:path_fan_reg_bound} is
$\min\{|V(G_1)|-1,m+k-1\}+t-1$ by \Cref{thm:FkW_pure_reg} and
\Cref{lem:reg_path}. Furthermore, we also allow the degenerate case $k=0$ case in \Cref{prop:fan_path_clique_sum} where
$G_1$ is the complete graph $K_n$.

\begin{proof}
    We prove by induction for $t\ge 2$. If  $t=2$, this is covered by \Cref{thm:FkW_pure_reg} since $G=F_{k+1}^{W\cup\{v\}}(K_{n})$. In particular, we have an equality at this step. In the following, we shall assume that $t\ge 3$. From \Cref{obs:fan_path_clique_sum} and induction it follows that
    \begin{align}
        \reg(S/\calJ_{K_m,G_v}) & \le \min\{|V(G_1)|,m+k-1\}+t-2
        \label{eqn:path_fan_reg_1} \\
        \intertext{and}
        \reg(S/\calJ_{K_m,G_v\setminus v}) & \le \min\{|V(G_1)|-1,m+k-1\}+t-2.
        \label{eqn:path_fan_reg_2}
    \end{align}
    Meanwhile,
    \[
        \reg(S/\calJ_{K_m,G\setminus v})=\min\{|V(G_1)|-2,m+k-1\}+t-2
    \]
    by \Cref{thm:FkW_pure_reg} and \Cref{lem:reg_path}. It follows from
    \Cref{depthlemma}\ref{depthlemma-c} that
    \begin{align*}
        \reg(S/\calJ_{K_m,G}) & \le \max\{|V(G_1)|+t-2,|V(G_1)|-1+t-2+1,|V(G_1)|-2+t-2\} \\
        & = |V(G_1)|+t-2,
        \intertext{and}
        \reg(S/\calJ_{K_m,G}) & \le \max\{m+k-1+t-2,m+k-1+t-2+1,m+k-1+t-2\} \\
        & =m+k+t-2.
    \end{align*}
    Therefore, the inequality in \eqref{eqn:path_fan_reg_bound} holds.

    In the rest of the proof, we assume additionally that $m\ge k+1$. Thus, by induction, we have equalities in \eqref{eqn:path_fan_reg_1} and \eqref{eqn:path_fan_reg_2}. If $|V(G_1)|\le m+k-1$, then
    \begin{gather*}
        \max\{
            \reg(S/\calJ_{K_m,G_v}),
            \reg(S/\calJ_{K_m,G\setminus v})
        \}=\max\{|V(G_1)|+t-2,|V(G_{1})|-2+t-2\}\\
        =|V(G_1)|+t-2 \ne
        \reg(S/\calJ_{K_m,G_v\setminus v})=|V(G_1)|-1+t-2.
    \end{gather*}
    Hence the equality in \eqref{eqn:path_fan_reg_bound}  holds by
    \Cref{depthlemma}\ref{depthlemma-c}.

    In the remaining part, we will assume that $|V(G_1)|\ge m+k$. Whence,
    \begin{align*}
        \reg(S/\calJ_{K_m,G}) & \ge
        a_{m+|V(G)|-1}(S/\calJ_{K_m,G})+ m+|V(G)|-1 \\
        &\ge (-|V(G)|+k+t-1)+m+|V(G)|-1=m+k+t-2,
    \end{align*}
    by \Cref{lem:fan_path_clique_sum}.
    Note that in this case the RHS of \eqref{eqn:path_fan_reg_bound} is also given by $m+k+t-2$.  Thus, we still get an equality here.
\end{proof}

In addition to studying the regularity directly, we also need information about the $a$-invariants.

\begin{Lemma}
    \label{lem:fan_path_clique_sum_base}
    Under the assumptions of \Cref{setting:fan_path_clique_sum}, we have
    \[
        a_{m+|V(G)|-1}(S/\calJ_{K_m,G})\ge a_{m+|V(G)|-1-(t-2)}(S/\calJ_{K_m,F_{k+1}^{W\cup\{v\}}(K_{n})}).
    \]
\end{Lemma}

\begin{proof}
    Note that $G_v\setminus v$ is isomorphic to $G'\coloneqq
    F_{k}^{W}(K_{n})\cup_v P_{t-1}$ with $|V(G')|=|V(G)|-1$.  It follows from
    \Cref{cor:compare_a_inv} that
    \[
        a_{m+|V(G)|-1}(S/\calJ_{K_m,G})\ge a_{m+|V(G')|-1}(S/\calJ_{K_m,G'}).
    \]
    By induction on $t$, we can obtain the desired inequality, since
    $F_{k+1}^{W\cup\{v\}}(K_{n})$ is isomorphic to $F_{k}^{W}(K_{n})\cup_v
    P_2$.
\end{proof}

\begin{Corollary}
    \label{lem:fan_path_clique_sum}
    Under the assumptions of \Cref{setting:fan_path_clique_sum}, we also assume that $|V(G_1)|\ge m+k$. Then, $a_{m+|V(G)|-1}(S/\calJ_{K_m,G})\ge -|V(G)|+k+t-1$.
\end{Corollary}
\begin{proof}
    Let $\widetilde{G}=F_{k+1}^{W\cup\{v\}}(K_{n})$. It is clear that $|V(\widetilde{G})|+t-2=|V(G_1)|+t-1=|V(G)|$. Therefore, it follows from \Cref{lem:fan_kn_v<m+k}\ref{lem:fan_kn_v<m+k_c} and \Cref{lem:fan_path_clique_sum_base} that
    \begin{align*}
        a_{m+|V(G)|-1}(S/\calJ_{K_m,G})
        & \ge a_{m+|V(\widetilde{G})|-1}(S/\calJ_{K_m,\widetilde{G}})
        \ge -|V(\widetilde{G})|+(k+1)\\
        &=-|V(G)|+k+t-1. \qedhere
    \end{align*}
\end{proof}

In the following, we present a subadditive formula when the $*$
operation is applied to a pseudo fan graph with a path graph.

\begin{Proposition}
    \label{prop:fan_path_star_reg}
    Let $G= G_1 *_f P_t$ such that $G_1$ is a $k$-pure pseudo fan graph,
    $|V(G_1)|-k\ge 3$, $t\ge 2$, and $k\ge 1$. Then we have
    \begin{align}
        \reg(S/\calJ_{K_m,G})\le\reg(S/\calJ_{K_m,G_1})
        +\reg(S/\calJ_{K_m,P_t}).
        \label{eqn:fan_path_star_reg-bound}
    \end{align}
    Furthermore, if $m\ge k$, then we have equality here.
\end{Proposition}
\begin{proof}
    It follows from \Cref{thm:FkW_pure_reg}, \Cref{lem:reg_path}, and \Cref{thm:reg_Fp} that the RHS of \eqref{eqn:fan_path_star_reg-bound} is given by
    \[
        t-1+\begin{cases}
            m+k-1,      & \text{if $m+k\le |V(G_1)|$,} \\
            |V(G_1)|-1, & \text{otherwise.}
        \end{cases}
    \]
    To prove the inequality in \eqref{eqn:fan_path_star_reg-bound}, let $v$ be the neighbor point of $f$ in $G_1$ and $H$ be the complete graph on the vertex set $N_G[v]$. It is not difficult to observe the following facts.
    \begin{enumerate}[a]
        \item Suppose that $G_1$ is the fan graph $F_k^W(K_n)$ with respect to $W=W_{1}\sqcup \cdots \sqcup W_{k}$. With loss of generality, we assume that $v\in W_1$. Then, $G_v$ is the clique sum $F_{k-1}^{W\setminus W_1}(H) \cup_f P_t$, $G_v\setminus v$ is the clique sum $F_{k-1}^{W\setminus W_1}(H\setminus v) \cup_f P_t$, and $f\in H\setminus v$. Furthermore, $G\setminus v$ is the disjoint union of $F_{q}^{W\setminus \{v\}}(K_{n-1})$ and $P_t$, where $q=k-1$ or $k$.
        \item Suppose that $G_1=F_p$. Since $k=2$ in this case, we have $p\ge 3$. Then, $G_v=F_1^{W'}(H)\cup_f P_t$ and $G_v\setminus v=F_1^{W'}(H\setminus v)\cup_f P_t$ such that $H=K_{p+1}$, $W'=N_{F_{p}}(v)\setminus f$, and $f\in H\setminus v$. Furthermore, $G\setminus v$ is the disjoint union of $F_{p-1}$ and $P_t$. In this case, we can set $q=2$ for later discussion.
    \end{enumerate}
    It follows from
    \Cref{prop:fan_path_clique_sum} that
    \begin{align}
        \reg(S/\calJ_{K_m,G_v}) & \le \reg(S/\calJ_{K_m,(G_1)_v})+\reg(S/\calJ_{K_m,P_t}) \label{eqn:Gv_reg_bound} \\
        & =\min\{|V(G_1)|-1,m+k-2\}+t-1,   \notag \\
        \intertext{and}
        \reg(S/\calJ_{K_m,G_v\setminus v}) & \le
        \reg(S/\calJ_{K_m,(G_1)_v\setminus v})+\reg(S/\calJ_{K_m,P_t}) \label{eqn:Gvv_reg_bound} \\
        & =\min\{|V(G_1)|-2,m+k-2\}+t-1.\notag
    \end{align}
    At the same time, we have
    \[
        \reg(S/\calJ_{K_m,G\setminus v})=\min\{|V(G_1)|-3,m+q-1\}+t-1
    \]
    by \Cref{thm:FkW_pure_reg}, \Cref{lem:reg_path}, and \Cref{thm:reg_Fp}.
    It follows from
    \Cref{depthlemma}\ref{depthlemma-c}
    that
    \begin{align*}
        \reg(S/\calJ_{K_m,G}) & \le \max\{|V(G_1)|-1,|V(G_1)|-3,|V(G_1)|-2+1\}+t-1
        \\
        & =|V(G_1)|-1+t-1.\notag
    \end{align*}
    For the same reason, we also have
    \begin{align*}
        \reg(S/\calJ_{K_m,G}) & \le
        \max\{m+k-2,m+q-1,m+k-2+1\}+t-1
        \\
        & =m+k-1+t-1,\notag
    \end{align*}
    since $q\le k$. The expected inequality in \eqref{eqn:fan_path_star_reg-bound} then follows.

    Finally, we assume that $m\ge k$ and show that an equality holds in \eqref{eqn:fan_path_star_reg-bound}. By \Cref{prop:fan_path_clique_sum}, we have equalities in both \eqref{eqn:Gv_reg_bound} and \eqref{eqn:Gvv_reg_bound}.  When $|V(G_1)|\le m+k-1$, then
    \begin{gather*}
        \max\{
            \reg(S/\calJ_{K_m,G_v}),
            \reg(S/\calJ_{K_m,G\setminus v})
        \}=\max\{|V(G_1)|+t-2,|V(G_{1})|+t-4\}\\
        =(|V(G_1)|-1)+t-1 \ne
        \reg(S/\calJ_{K_m,G_v\setminus v})=(|V(G_1)|-2)+t-1.
    \end{gather*}
    Hence $\reg(S/\calJ_{K_m,G})=(|V(G_1)|-1)+t-1$ by \Cref{depthlemma}\ref{depthlemma-c}, which establishes the equality in \eqref{eqn:fan_path_star_reg-bound} in this case.

    On the other hand, when $|V(G_1)|\ge m+k$, we can still prove that the equality  in \eqref{eqn:fan_path_star_reg-bound} holds. Indeed, since $G_v\setminus v = F_{k-1}^{W\setminus W_1}(H\setminus v) \cup_f P_t$, $|V(F_{k-1}^{W\setminus W_1}(H\setminus v))|=|V(G_1)|-1\ge m+(k-1)$, and $|V(G)|=|V(G_v\setminus v)|+1$ in this case, we deduce from \Cref{cor:compare_a_inv} and \Cref{lem:fan_path_clique_sum} that
    \begin{align*}
        \reg(S/\calJ_{K_m,G}) & \ge
        a_{m+|V(G)|-1}(S/\calJ_{K_m,G})+ m+|V(G)|-1 \\
        &\ge a_{m+|V(G_v\setminus v)|-1}(S/\calJ_{K_m,G_v\setminus v})+
        m+|V(G)|-1 \\
        & \ge -|V(G_v\setminus v)|+(k-1)+t-1 +
        m+|V(G)|-1 \\
        & = m+k+t-2.
    \end{align*}
    Note that the RHS of \eqref{eqn:fan_path_star_reg-bound} is also given by $m+k+t-2$. So we still have an equality here.
\end{proof}

The following result provides a rare additive formula when applying the $*$ operation.

\begin{Corollary}
    \label{prop:fp_path_star_reg}
    Let $G= F_{p}*_f P_t$ with $p\ge 1$ and $t\ge 2$. Then, we have
    \begin{align*}
        \reg(S/\calJ_{K_m,G}) & =\reg(S/\calJ_{K_m,F_p})+\reg(S/\calJ_{K_m,P_t}) \\
        & = \min\{2p-1,m+1\}+t-1.
    \end{align*}
\end{Corollary}
\begin{proof}
    If $p=1$ or $2$, it is clear that $G=P_{2p-1+t}$. Since $m\ge 2$, the desired formula follows from \Cref{lem:reg_path}. If instead $p\ge 3$,
    since $k=2$  and $|V(F_p)|-2\ge 3$, it remains to apply \Cref{prop:fan_path_star_reg}.
\end{proof}

\begin{Corollary}
    \label{prop:fan_path_circ_reg}
    Let $G= F_{k}^{W}(K_n)\circ_v P_t$ with $n\ge 3$. Then, we have
    \[
        \reg(S/\calJ_{K_m,G})\le \reg(S/\calJ_{K_m,F_{k}^{W}(K_n)})+\reg(S/\calJ_{K_m,P_t})-s,
    \]
    where $s= \begin{cases}
        1, & \text{if $t=2$},\\
        2, & \text{if $t\ge 3$.}
    \end{cases}$
\end{Corollary}
\begin{proof}
    Let $G_1=F_{k}^{W}(K_n)$ and $f$ be the leaf in $G_1$ whose neighbor point is $v$.
    \begin{enumerate}[a]
        \item If $t=2$, then $G=G_1\setminus f$. In particular, $\reg(S/\calJ_{K_m,G})\le \reg(S/\calJ_{K_m,G_1})$  by \Cref{cor:induced_graph}. On the other hand, $\reg(S/\calJ_{K_m,P_2})=1$ by \Cref{lem:reg_path}.
            The claimed inequality is clear in this case.

        \item If $t=3$, then $G=G_1$. The expected inequality is trivial, since $\reg(S/\calJ_{K_m,P_3})=2$.

        \item If $t\ge 4$, then 
            $G=F_{k}^{W}(K_n)*P_{t-2}$. The expected result follows from  \Cref{prop:fan_path_star_reg}, since $\reg(S/\calJ_{K_m,P_t})=\reg(S/\calJ_{K_m,P_{t-2}})+2$. \qedhere
    \end{enumerate}
\end{proof}

From the proofs of \Cref{prop:fan_path_star_reg} and  \Cref{prop:fp_path_star_reg}, it is clear that the clique sums are important for our study of $*$ operations. Therefore,  in the following we will study the clique sums of more complicated combinations.

\begin{Proposition}
    \label{prop:clique_sum_reg}
    Let $G=G_1 \cup_u G_2$  be the clique sum of two graphs $G_1$ and $G_2$.  Suppose that $G_{2}=F_{k_2}^{W_2}(K_{n_2})$ is a $k_{2}$-pure fan graph of $K_{n_2}$ on $W_2$ with $|V(G_2)|-k_2\ge 3$ and $u\in V(K_{n_2})\setminus W_2$, and $G_1$ is a $k_1$-pure pseudo fan graph with $|V(G_1)|-k_1\ge 3$ and $k_1\ge 1$. Suppose in addition that $u$ is a leaf of $G_1$ and $v$ is the unique neighbor point of $u$ in $G_1$.  Then, we have
    \[
        \reg(S/\calJ_{K_m,G})\le \min\{ (m+k_1-1)+(m+k_2-1), m+k_2+|V(G_1)|-2\}.
    \]
\end{Proposition}
\begin{proof}
    Let $H$ be the complete graph on the vertex set $N_G[v]$. Then, $G_v$ can be regarded as a graph obtained by a $\circ$ operation from the fan graphs $F_{k_1}^{W'_1}(H)$ on some $W'_1$  and $F_{k_2+1}^{W'_2}(K_{n_2})$ on some $W'_2$. Note that the numbers of vertices satisfy $|V(F_{k_1}^{W'_1}(H))|=|V(G_1)|+1$ and $|V(F_{k_2+1}^{W'_2}(K_{n_2}))|=|V(G_2)|+1$. Thus,
    by \Cref{prop:reg_Fan_Fan_circ}  we obtain that
    \begin{align*}
        \reg(S/\calJ_{K_m, G_v}) & \le \max\big\{
            \min\{|V(G_1)|-2,m+(k_1-1)-1\}+\min\{|V(G_2)|-2,m+k_2-1\}, \\
            & \qquad \qquad \min\{|V(G)|-1,m+k_1+(k_2+1)-3\}, \\
            & \qquad \qquad \min\{|V(G)|-1,m+k_1+(k_2+1)-2\}
        \big\}.
    \end{align*}
    Meanwhile, $G\setminus v$ is the disjoint union of a $q_1$-pure pseudo
    fan graph $G_1\setminus\{u,v\}$  and $G_2$, where $q_1=k_1-1$ or $k_1$.
    Thus  by \Cref{thm:FkW_pure_reg} we obtain that
    \[
        \reg(S/\calJ_{K_m,G\setminus
        v})=\min\{|V(G_1)|-3,m+q_1-1\}+\min\{|V(G_2)|-1,m+k_2-1\}.
    \]

    We can draw two conclusions from these observations.  First, since $m\ge 2$, we have
    \begin{align*}
        \reg(S/\calJ_{K_m,G_v}) & \le \max\{2m+k_1+k_2-3,m+k_1+k_2-2,m+k_1+k_2-1\} \\
        & =2m+k_1+k_2-3.
    \end{align*}
    Meanwhile, we obtain that
    \[
        \reg(S/\calJ_{K_m,G\setminus v})
        \le (m+q_1-1)+(m+k_2-1)
        \le 2m+k_1+k_2-2.
    \]
    Thus, after applying \Cref{lem:reg_easy_bound}, we have
    $\reg(S/\calJ_{K_m,G})\le (m+k_1-1)+(m+k_2-1)$.

    Second,
    we also have
    \begin{align*}
        \reg(S/\calJ_{K_m,G_v}) & \le \max\{|V(G_1)|-2+m+k_2-1,m+k_1+k_2-2,m+k_1+k_2-1\} \\
        & \le m+k_2+|V(G_1)|-3
    \end{align*}
    since $|V(G_1)|-k_1\ge 3$.
    On the other hand, we have
    \[
        \reg(S/\calJ_{K_m, G\setminus v})
        \le (m+k_2-1)+(|V(G_1)|-3).
    \]
    Thus, after applying \Cref{lem:reg_easy_bound}, we have
    $\reg(S/\calJ_{K_m,G})\le m+k_2+|V(G_1)|-2$.
    \qedhere
\end{proof}

\begin{Remark}
    In \Cref{prop:clique_sum_reg}, if instead
    $G_1=F_1$, then
    $G=F_{k_2+1}^{W\cup\{u\}}(K_n)$ is a $(k_{2}+1)$-pure fan graph. It
    follows from \Cref{thm:FkW_pure_reg} that $\reg(S/\calJ_{K_m,G})=\min\{|V(G)|-1,m+k_2\}$. On the other hand, if $G_1=F_2$, then
    $G=F_{k_2+1}^{W\cup\{u\}}(K_n)*_v P_3$. Thus,  by \Cref{thm:FkW_pure_reg} and
    \Cref{prop:fan_path_star_reg}, we obtain
    $\reg(S/\calJ_{K_m,G})\le m+k_2+2$.
\end{Remark}

We are now ready to give the first main result of this subsection.

\begin{Theorem}
    \label{thm:*_lower_upper_bound}
    Let $G=(G_1,f_1) * (G_2,f_2)$, where $G_i$ is a $k_i$-pure pseudo fan graph with $|V(G_i)|-k_i\ge 3$ and $k_i\ge 1$ for each $i$. Let $v_i$ be the neighbor point of $f_i$ in $G_i$ for $i=1,2$, and $f=f_1=f_2$ in $G$.  Then we have
    \begin{equation}
        \sum\limits_{i=1}^2\reg(S/\calJ_{K_m,G_i\setminus \{f_i,v_i\}})\le \reg(S/\calJ_{K_m,G})\le\sum\limits_{i=1}^2\reg(S/\calJ_{K_m,G_i}).
        \label{eqn:*_lower_upper_bound}
    \end{equation}
\end{Theorem}

\begin{proof}
    The first inequality of \eqref{eqn:*_lower_upper_bound} is trivial by \Cref{cor:induced_graph}. In the following, we deal with  the second inequality of \eqref{eqn:*_lower_upper_bound}.
    For simplicity, let $A\coloneqq \min\{|V(G_i)|-k_i: i\in [2]\}$ and $B\coloneqq \max\{|V(G_i)|-k_i : i\in [2]\}$.
    Without loss of generality, we can suppose that $A=|V(G_1)|-k_1$ and $B=|V(G_2)|-k_2$.
    By \Cref{thm:FkW_pure_reg},   \Cref{thm:reg_Fp} and the fact that $|V(G)|=|V(G_1)|+|V(G_2)|-1$, we can get
    \begin{align}
        \sum\limits_{i=1}^2\reg(S/\calJ_{K_m,G_i})
        & =\min\{|V(G_1)|-1,m+k_1-1\}+\min\{|V(G_2)|-1,m+k_2-1\} \notag \\
        &=\begin{cases}
            2m+k_1+k_2-2,     & \text{if $m\le A$,}        \\
            m+k_2+|V(G_1)|-2, & \text{if $A+1\le m\le B$,} \\
            |V(G)|-1,         & \text{if $ m\ge B+1$.}
        \end{cases}
        \label{eqn:*_lower_upper_bound_right}
    \end{align}
    Since $G_{v_2}=G_1 \cup_f (G_{2})_{v_2}$ is the clique sum of a $k_1$-pure  pseudo
    fan graph and a $(k_2-1)$-pure fan graph satisfying the conditions in \Cref{prop:clique_sum_reg},
    we can get
    \[
        \reg(S/\calJ_{K_m,G_{v_2}})\le (m+k_1-1)+(m+k_2-2).
    \]
    Since $G\setminus v_2$ is the disjoint union of the graph $G_1$ 
    and a $q_2$-pure pseudo fan graph such that $q_2\le k_2$, we have
    \begin{align}
        \reg(S/\calJ_{K_m,G\setminus v_2}) & =\min\{|V(G_1)|-1,m+k_1-1\}+\min\{|V(G_2)|-3,m+q_2-1\}
        \label{eqn:GDv1_formula} \\
        & \le (m+k_1-1)+(m+q_2-1) \notag
    \end{align}
    by \Cref{thm:FkW_pure_reg} and   \Cref{thm:reg_Fp}.  So we have $\reg(S/\calJ_{K_m,G})\le
    2m+k_1+k_2-2$ by \Cref{lem:reg_easy_bound}.  Consequently,
    \[
        \reg(S/\calJ_{K_m,G})\le \min\{2m+k_1+k_2-2,|V(G)|-1\}
    \]
    by \Cref{lem:reg_min_equal}.

    Therefore, in view of
    \eqref{eqn:*_lower_upper_bound_right}, it suffices to assume that $A+1\le
    m\le B$ and prove that $\reg(S/\calJ_{K_m,G})\le m+k_2+|V(G_1)|-2$. Note that
    \[
        \reg(S/\calJ_{K_m,G_{v_2}})\le m+(k_2-1)+|V(G_1)|-2
    \]
    by \Cref{prop:clique_sum_reg}.
    Meanwhile,
    \[
        \reg(S/\calJ_{K_m,G\setminus v_2}) \le (|V(G_1)|-1)+(m+q_2-1) \le m+k_2+|V(G_1)|-2
    \]
    by \eqref{eqn:GDv1_formula}.  The desired result then follows from \Cref{lem:reg_easy_bound},  which completes  our proof.
\end{proof}

From the proof of \Cref{thm:*_lower_upper_bound}, we can make the following quick estimates.

\begin{Corollary}
    \label{cor:F*F}
    Under the assumptions of \Cref{thm:*_lower_upper_bound}, we have the following estimates:
    \begin{enumerate}[a]
        \item   $\reg(S/\calJ_{K_m,F_{k_1}^{W_1}(K_{n_1})*F_{k_2}^{W_2}(K_{n_2})})\le
            2m+k_1+k_2-2$,
        \item  $\reg(S/\calJ_{K_m,F_{n_1}*F_{k_2}^{W_2}(K_{n_2})})\le 2m+k_2$, and
        \item  \label{cor:F*F_c}
            $\reg(S/\calJ_{K_m,F_{n_1}*F_{n_2}})\le 2m+2$.
    \end{enumerate}
\end{Corollary}

\begin{Corollary}
    \label{cor:*_add}
    Let $n_1, n_2\ge 3$ be integers. If $m\le \min\{2n_1-2,2n_2-4\}$, then
    \[
        \reg(S/\calJ_{K_m,F_{n_1}*_f F_{n_2}})=\sum_{i=1}^2
        \reg(S/\calJ_{K_m,F_{n_i}})=2m+2.
    \]
\end{Corollary}
\begin{proof}
    Suppose that $v_2$ is the neighbor point of $f$ in $G_2$. 
    Note that
    $G\setminus v_2=F_{n_1}\sqcup F_{n_2-1}$. If $m\le \min\{2n_1-2,2n_2-4\}$,
    then $\reg(S/\calJ_{K_m,G\setminus v_2})=2m+2$ by \Cref{thm:reg_Fp}. Since
    $G\setminus v_2$ is an induced subgraph of $G$, this implies that
    $\reg(S/\calJ_{K_m,G})\ge 2m+2$ by \Cref{cor:induced_graph}. At the same
    time, we have $\reg(S/\calJ_{K_m,G}) \le 2m+2$ from \Cref{cor:F*F}.
    Consequently, $\reg(S/\calJ_{K_m,G}) =2m+2$. Meanwhile, it is easy to
    check that $\sum_{i=1}^2 \reg(S/\calJ_{K_m,F_{n_i}}) =2m+2$ if $m\le
    \min\{2n_1-2,2n_2-4\}$.
\end{proof}

\begin{Example}
    \label{exam:333_star}
    Let $m$ be $3$ or $4$, and $G=(G_1,f_1)*(G_2,f_2)$ with $G_1=G_2=F_3$. In this case, the condition of \Cref{cor:*_add} is not satisfied. Indeed, we claim that $\reg(S/\calJ_{K_m,F_3*F_3})\le 8$.
    To confirm this claim, let $v_i$ be the neighbor point of $f_i$ in $G_i$ for each $i$.
    By \Cref{thm:FkW_pure_reg}, \Cref{thm:reg_Fp} and \Cref{prop:reg_Fan_Fan_circ}, we obtain that $\reg(S/\calJ_{K_m,(G_{v_2})_{v_1}}) \le 2m$,   $\reg(S/\calJ_{K_m,(G_{v_2})_{v_1}\setminus v_1})\le m+3$, and $\reg(S/\calJ_{K_m,G_{v_2}\setminus v_1}) \le m+3$. It follows from \Cref{depthlemma}\ref{depthlemma-c} that $\reg(S/\calJ_{K_m,G_{v_2}})\le 8$.  Similarly,  we can get $\reg(S/\calJ_{K_m,(G_{v_2}\setminus v_2)_{v_1}}) \le m+3$,  $\reg(S/\calJ_{K_m,(G_{v_2}\setminus v_2)_{v_1}\setminus v_1})\le 6$, and $\reg(S/\calJ_{K_m,(G_{v_2}\setminus v_2)\setminus v_1}) \le m+3$. It follows from \Cref{depthlemma}\ref{depthlemma-c} that $\reg(S/\calJ_{K_m,G_{v_2}\setminus v_2})\le 7$. Finally, we have $\reg(S/\calJ_{K_m, G\setminus v_2})=m+4$. 
    Thus, \Cref{depthlemma}\ref{depthlemma-c} and the exact sequence (\ref{eqn:SES-1}) in \Cref{rem:exactseq} imply that $\reg(S/\calJ_{K_m,G})\le 8$, as claimed.

    In particular, the inequality in \Cref{cor:F*F}\ref{cor:F*F_c} is strict for $(m,n_1,n_2)=(4,3,3)$. For $(m,n_1,n_2)=(3,3,3)$, it is not yet clear whether the inequality in \Cref{cor:F*F}\ref{cor:F*F_c} is strict or not.
\end{Example}

\begin{Question}
    Does the second inequality of \eqref{eqn:*_lower_upper_bound} hold for the general graphs $G_1$ and $G_2$?
\end{Question}

\subsection{Main result}
In the following subsection, we combine all the graphs and operations that we
have previously considered. Let us begin with a useful lemma.

\begin{Lemma}
    \label{lem:reg_easy_bound_star}
    Suppose that $G=(G_1,f_1)*(G_2,f_2)$. Let $v_i$ be the neighbor point of
    $f_i$ in $G_i$ for $i=1,2$. Suppose also that $f_1$ and $f_2$ are
    identified as $f$ in $G$. Then,
    \begin{align*}
        \reg(S/\calJ_{K_m,G}) \le \max\{ &
            \reg(S/\calJ_{K_m,G\setminus v_1}),
            \reg(S/\calJ_{K_m, G_{v_1}\setminus v_2})+1, \\
            & \reg(S/\calJ_{K_m,(G_{v_1})_{v_2}\setminus f})+2,
            \reg(S/\calJ_{K_m, ((G_{v_1})_{v_2})_{f}})+3
        \}.
    \end{align*}
\end{Lemma}
\begin{proof}
    It suffices to apply \Cref{lem:reg_easy_bound} successively to
    $(G_{v_1})_{v_2}$ with respect to $f$, $G_{v_1}$ with respect to $v_2$,
    and $G$ with respect to $v_1$.
\end{proof}

It is time for the main result of this section. In the following theorem, we
also involve the path graphs in the discussion to make the final result
slightly sharper.  Note that the path graph $P_t=\underbrace{P_2* \cdots *
P_2}_{t-1}$ while $P_2$ is simply $F_1$ and $P_4$ is $F_2$.

\begin{Theorem}
    \label{thm:full_reg_star_circ_path}
    Let $N$ be a positive integer. Suppose that $G=G_1 \circledast_1 \cdots
    \circledast_{N-1} G_N$ where
    each $G_i$ is either the path graph $P_{n_i}$ with $n_i\ge 2$, or a
    $k_i$-pure pseudo fan graph with  $|V(G_i)|\ge 3$.
    Furthermore, each $\circledast_i$ is either a $*$ operation or a $\circ$
    operation using only marked leaves.
    Then, we have
    \begin{equation}
        \reg(S/\calJ_{K_m,G}) \le \sum\limits_{i=1}^{N}\ell_i,
        \label{eqn:main_result}
    \end{equation}
    where $\ell_i\coloneqq \begin{cases}
        n_i-1,   & \text{if $G_i=P_{n_i}$,} \\
        m+k_i-1, & \text{if $G_i$ is a $k_i$-pure pseudo fan
        graph.}
    \end{cases}$
\end{Theorem}
\begin{proof}
    If some $G_i$ is a $k_i$-pure pseudo fan graph with $|V(G_i)|-k_i\le 2$,
    we can replace this $G_i$ with a $k_i$-pure pseudo fan graph $G_i'$, such
    that $|V(G_i')|-k_i\ge 3$, and $G_i$ is an induced subgraph of $G_i'$ as
    marked graphs. With this reduction, in the subsequent proof, we can 
    freely use the results such as 
    \Cref{cor:reg_Fan_fan_upper_bound}, \Cref{prop:fan_path_star_reg} and
    \Cref{thm:*_lower_upper_bound}, without checking the $|V(G_i)|-k_i\ge 3$
    condition.

    Now, we prove the assertion in \eqref{eqn:main_result} by applying a
    two-stage induction. The first stage of induction is with respect to $N$.
    When $N=1$, it follows from \Cref{thm:FkW_pure_reg}, \Cref{lem:reg_path},
    and \Cref{thm:reg_Fp} respectively. When $N=2$, it follows from
    \Cref{lem:reg_path}, \Cref{cor:reg_Fan_fan_upper_bound},
    \Cref{prop:fan_path_star_reg}, \Cref{prop:fan_path_circ_reg} and
    \Cref{cor:F*F}. In the following, we assume that $N\ge 3$. Suppose that
    the operation $\circledast_1$ is performed with respect to the marked
    leaves $f_1\in G_1$ and $f_2\in G_2$. Let $v_i$ be the neighbor point of
    $f_i$ in $G_i$ for $i=1,2$. We distinguish into the following four cases.
    \begin{enumerate}[1]
        \item First, suppose that $G_i= P_{n_i}$ for $i=1,2$. Then, $G=P_n
            \circledast_2 G_3\circledast_3\cdots \circledast_{N-1} G_N$, where
            $n=n_1+n_2-3$ if $\circledast_1=\circ$, or $n=n_1+n_2-1$ if
            $\circledast_1=*$. So the expected result follows from the
            inductive hypothesis.
        \item Suppose next that $G_1= P_{n_1}$ and $G_2\ne P_{n_2}$. 
            If
            $\circledast_1=\circ$ and $n_1=2$, then $G$ is an induced subgraph of $G_2 \circledast_2 \cdots
            \circledast_{N-1} G_N$. In this case, the expected inequality follows from induction hypothesis and \Cref{cor:induced_graph}.
            If instead
            $\circledast_1=\circ$ and $n_1\ge 3$, then
            $G=P_{n_1-2} * G_2\circledast_2\cdots \circledast_{N-1} G_N$. 

            Therefore, it suffices  to prove the assertion in the case where
            $G=P_{n_1} * G_2\circledast_2\cdots \circledast_{N-1} G_N$.
            The second stage induction is then with respect to
            $n_1\ge 1$.
            The degenerate case when $n_1=1$ follows from the inductive
            hypothesis, since $G=G_2\circledast_2\cdots \circledast_{N-1} G_N$
            in this case. We can therefore assume that $n_1\ge
            2$.
            It is not difficult to observe the following.
            \begin{itemize}
                \item The graph $G\setminus v_2$ is the disjoint union of
                    $P_{n_1}$ and $(G_2\setminus \{f_2,v_2\})\circledast_2 G_3
                    \circledast_3 \cdots \circledast_{N-1} G_N$. Note that the
                    latter is an induced subgraph of $G_2\circledast_2 G_3
                    \circledast_3 \cdots \circledast_{N-1} G_N$.

                \item \label{proof:full_reg_star_circ_path_2}
                    The graph $G_{v_2}=H\circledast_2 G_3\circledast_3\cdots
                    \circledast_{N-1} G_N$, where $H$ is a graph obtained from
                    $P_{n_1-1}$ and a $(k_2-1)$-pure fan graph by a $*$
                    operation. 
            \end{itemize}
            Therefore, it follows from \Cref{cor:induced_graph}, and
            \Cref{lem:reg_path}, and the inductive hypothesis with respect to
            $N$ and $n_1$ that
            \begin{align*}
                \reg(S/\calJ_{K_m,G\setminus v_2})&\le \ell_1+ \sum_{i=2}^N \ell_i,
                \intertext{and}
                \reg(S/\calJ_{K_m,G_{v_2}})&\le
                (\ell_1-1)+(\ell_2-1)+\sum_{i=3}^N\ell_i.
            \end{align*}
            It remains to apply \Cref{lem:reg_easy_bound}.
        \item Assume symmetrically that  $G_1\ne P_{n_1}$ and $G_2=P_{n_2}$.
            Likewise, we can suppose that $\circledast_1=*$.
            By applying similar arguments as in
            \ref{proof:full_reg_star_circ_path_2} with respect to the point
            $v_1$, we can obtain the desired result.
        \item Finally, suppose that neither $G_1$ nor $G_2$ is a path
            graph.
            In this case, we consider the following two subcases.
            \begin{enumerate}[i]
                \item Suppose that $\circledast_1=\circ$. Assume that $v_1$
                    and $v_2$ are identified as $v$ in $G$. It is not
                    difficult to observe the following.
                    \begin{itemize}
                        \item The graph $G\setminus v$ is the disjoint union
                            of $G_1\setminus \{f_1,v_1\}$ and $(G_2\setminus
                            \{f_2,v_2\})\circledast_2 G_3 \circledast_3 \cdots
                            \circledast_{N-1} G_N$. Note that $G_1\setminus
                            \{f_1,v_1\}$ and $G_2\setminus \{f_2,v_2\}$ are
                            induced subgraphs of $G_1$ and $G_2$ respectively.

                        \item The graph $G_v$ is obtained from a
                            $(k_1+k_2-2)$-pure fan graph and $G_3
                            \circledast_3 \cdots \circledast_{N-1} G_N$ by a
                            $\circledast_2$ operation.
                    \end{itemize}
                    So from \Cref{cor:induced_graph} and the
                    inductive hypothesis with respect to $N$, we obtain that
                    \begin{align*}
                        \qquad \qquad \reg(S/\calJ_{K_m,G\setminus v})& \le
                        \sum\limits_{i=1}^{N}\ell_i,
                        \intertext{and}
                        \reg(S/\calJ_{K_m,G_v}) &\le
                        (m+(k_1+k_2-2)-1)+\sum\limits_{i=3}^{N}\ell_i
                        =-(m+1)+\sum\limits_{i=1}^{N}\ell_i.
                    \end{align*}
                    It remains to apply \Cref{lem:reg_easy_bound}.
                \item Suppose that $\circledast_1=*$. Assume that $f_1$ and
                    $f_2$ are identified as $f$ in $G$.
                    It is not difficult to observe the following.
                    \begin{itemize}
                        \item The graph $G\setminus v_{1}$ is the disjoint
                            union of $G_1\setminus \{f_1,v_1\}$ and
                            $G_2\circledast_2 \cdots \circledast_{N-1} G_N$.
                        \item The graph $G_{v_{1}}\setminus v_{2}$ is the
                            disjoint union of a $(k_1-1)$-pure fan graph and
                            $(G_2\setminus \{f_2,v_2\})\circledast_2 G_3
                            \circledast_3 \cdots \circledast_{N-1} G_N$.
                        \item The graph $(G_{v_1})_{v_2}\setminus f$ is the
                            disjoint union of a $(k_1-1)$-pure fan graph and a
                            graph, which is obtained from a $(k_2-1)$-pure fan
                            graph and $G_3 \circledast_3 \cdots
                            \circledast_{N-1} G_N$ by the $\circledast_2$
                            operation.
                        \item The graph $((G_{v_1})_{v_2})_f$ is a graph
                            obtained from a $(k_1+k_2-2)$-pure fan graph and
                            $G_3 \circledast_3 \cdots \circledast_{N-1} G_N$
                            by a $\circledast_2$ operation.
                    \end{itemize}
                    Thus, by \Cref{cor:induced_graph},
                    \Cref{thm:FkW_pure_reg}, \Cref{thm:reg_Fp}, and the
                    inductive hypothesis with respect to $N$, one has
                    \begin{align*}
                        \qquad \qquad \qquad \qquad
                        \reg(S/\calJ_{K_m,G\setminus v_{1}}) & \le (m+ k_1-1)
                        +\sum_{k=2}^{N}\ell_i, \\
                        \reg(S/\calJ_{K_m,G_{v_{1}}\setminus v_{2}}) & \le
                        (m+(k_1-1)-1) +\sum_{k=2}^{N}\ell_i, \\
                        \reg(S/\calJ_{K_m,(G_{v_1})_{v_2}\setminus f}) & \le
                        (m+(k_1-1)-1)+(m+(k_2-1)-1)+\sum_{k=3}^N\ell_i,
                        \intertext{and}
                        \reg(S/\calJ_{K_m,((G_{v_1})_{v_2})_f}) & \le
                        (m+(k_1+k_2-2)-1)+\sum_{k=3}^{N}\ell_i.
                    \end{align*}
                    It remains to apply \Cref{lem:reg_easy_bound_star} by
                    noting that $m\ge 2$. \qedhere
            \end{enumerate}
    \end{enumerate}
\end{proof}

The following result deals with a special case of
\Cref{thm:full_reg_star_circ_path}. It gives a sharper result, which can be
seen as a generalization of \cite[Theorem 4.6]{MR3991052}. Note that
\cite[Theorem 4.6]{MR3991052} contains an error due to incorrect induction, 
which was later corrected in Arvind Kumar's thesis.


\begin{Corollary}
    \label{thm:multiple_fan_circ}
    Let $m,t\ge 2$ be integers. Suppose that $G=G_1 \circ G_2 \circ \cdots
    \circ G_t$ where each $G_i$ is a $k_i$-pure pseudo fan
    graph with $|V(G_i)|-k_i\ge 1$. For each $i\in [t]$ and $j\in
    [2]$, let $f_{i,j}$ be a marked leaf of $G_i$ and $v_{i,j}$ be its
    neighbor such that  $v_{i,2}$ and $v_{i+1,1}$ are identified
    in $G$ by the $\circ$ operation. Let $\widetilde{G_1}\coloneqq
    G_1\setminus \{v_{1,2},f_{1,2}\}$, $\widetilde{G_t}\coloneqq G_t\setminus
    \{v_{t,1},f_{t,1}\}$ and $\widetilde{G_i}\coloneqq G_i\setminus
    \{v_{i,1},f_{i,1},v_{i,2},f_{i,2}\}$ for all $i=2,\ldots, t-1$.  In
    addition,  suppose that $\widetilde{G_i}$ is a $q_i$-pure pseudo fan graph.
    Then, we have the following results.
    \begin{enumerate}[A]
        \item Firstly, we have
            \begin{align}
                \sum_{i=1}^t \reg(S/\calJ_{K_m,\widetilde{G_i}})
                & \le \reg(S/\calJ_{K_m,G}) \le \sum\limits_{i=1}^{t} (m+k_i-1).
                \label{eqn:multiple_fan_circ-2}
            \end{align}
        \item \label{thm:multiple_circ_b}
            Secondly, suppose that the following technical conditions are satisfied:
            \begin{enumerate}[a]
                \item \label{thm:multiple_circ_b_a}
                    if $G_i=F_{k_i}^{W_i}(K_{n_i})$ is a $k_i$-pure fan graph,
                    then $|V(G_i)|-k_i\ge 4$ when $i=1$ or $t$, and
                    $|V(G_i)|-k_i\ge 6$ when $2\le i\le t-1$;
                \item \label{thm:multiple_circ_b_b}
                    if $G_i=F_{n_i}$, then $n_i\ge 2$ when $i=1$, $n_i\ge 3$
                    when $i=t$, and $n_i\ge 4$ when $2\le i\le t-1$;
                \item \label{thm:multiple_circ_b_c}
                    for each $i$, one has
                    $k_i=q_i$;
                \item \label{thm:multiple_circ_b_d}
                    for each $i$, one has $m\le r_i\coloneqq
                    \begin{cases}
                        |V(G_i)|-k_i, & \text{if $i=1$ and $t$,}\\
                        |V(G_i)|-k_i-4, & \text{if $2\le i\le t-1$.}
                    \end{cases}$
            \end{enumerate}
            Then, the first inequality in \eqref{eqn:multiple_fan_circ-2}
            becomes an equality. Namely, we have
            \[
                \sum_{i=1}^t \reg(S/\calJ_{K_m,\widetilde{G_i}}) =
                \reg(S/\calJ_{K_m,G}).
            \]
    \end{enumerate}
\end{Corollary}
\begin{proof}
    \begin{enumerate}[A]
        \item Since $\widetilde{G_1}\sqcup\cdots\sqcup \widetilde{G_t}$ forms
            an induced subgraph of $G$, the first inequality in
            \eqref{eqn:multiple_fan_circ-2} follows from
            \Cref{cor:induced_graph}. As for the second inequality, we can
            use \Cref{thm:full_reg_star_circ_path}.

        \item In the following, we prove the assertion by induction on $t$.

            The base case is when $t=2$.
            \begin{enumerate}[i]
                \item If $G_1=F_{2}$, then $G=F_1*G_2$. Furthermore, $m=2$ by the assumptions
                    that $m\ge 2$ and $m\le r_1=|V(G_1)|-k_1=2$.  Thus, by
                    \cite[Theorem 3.1]{MR3941158} and \Cref{lem:reg_path},
                    we obtain $\reg(S/\calJ_{K_m,G})=\reg(S/J_G)=
                    \reg(S/J_{F_1})+ \reg(S/J_{G_2})$.
                    On the one hand, it is
                    clear that $\widetilde{G_1}=F_1$. On the other hand,
                    $\reg(S/J_{G_2})=\reg(S/J_{\widetilde{G_2}})$ by
                    \cite[Theorem 3.4]{MR3991052}. The claimed equality
                    is now obvious.
                \item Otherwise, $G_1$ is a $k_1$-pure pseudo fan graph with
                    $|V(G_1)|-k_1\ge 4$, by the assumption
                    \ref{thm:multiple_circ_b}\ref{thm:multiple_circ_b_a} or
                    \Cref{rmk:reg_Fan_Fan_circ}. The desired result is then
                    covered by \Cref{rem:reg_fan_fan}.
            \end{enumerate}
            This completes our proof for the base case.

            Next, we assume that $t\ge 3$. In this case, suppose that
            $v_{t-1,2}$ and $v_{t,1}$ are identified as $v$ in $G$. Then, we
            have $G\setminus v= (G_1 \circ \cdots \circ G_{t-2}\circ
            \widetilde{G_{t-1}}') \sqcup \widetilde{G_t}$, where $
            \widetilde{G_{t-1}}'\coloneqq G_{t-1}\setminus
            \{v_{t-1,2},f_{t-1,2}\}$. Since $m\le |V(G_{t-1})|-k_{t-1}-4$ by
            the assumption, we have $\reg(S/\calJ_{K_m, \widetilde{G_{t-1}}'
            \setminus \{v_{t-1,1},f_{t-1,1}\}}) = \reg(S/\calJ_{K_m,
        \widetilde{G_{t-1}}}) =m+k_{t-1}-1$ from \Cref{thm:FkW_pure_reg} and
        \Cref{thm:reg_Fp}. Furthermore, as $m\le \min\{r_1,\dots,r_t\}$, one
        has $m\le \min\{r_1,\dots,r_{t-2}, r'_{t-1}\}$, where $r'_{t-1}
        \coloneqq |V(\widetilde{G_{t-1}}')|-k_{t-1} = |V(G_{t-1})|-2-k_{t-1} =
        r_{t-1}+2$. The conditions of \ref{thm:multiple_circ_b_a},
        \ref{thm:multiple_circ_b_b}, and \ref{thm:multiple_circ_b_c} can 
        be checked in a similar way. Thus, by induction, we have
        \begin{align*}
            \qquad\qquad \reg(S/\calJ_{K_m,{G_1 \circ \cdots \circ
                    G_{t-2}\circ
            \widetilde{G_{t-1}}'}})&=\sum_{i=1}^{t-2}\reg(S/\calJ_{K_m,
            \widetilde{G_i}})+\reg(S/\calJ_{K_m,\widetilde{G_{t-1}}'
        \setminus \{v_{t-1,1},f_{t-1,1}\}})\\
        &=\sum_{i=1}^{t-1}\reg(S/\calJ_{K_m,\widetilde{G_i}}).
    \end{align*}
    Consequently, we have
    \begin{equation}
        \reg(S/\calJ_{K_m,G\setminus
        v})=\sum_{i=1}^{t}\reg(S/\calJ_{K_m,\widetilde{G_i}}).
        \label{eqn:cor_final_bound_1}
    \end{equation}

    Next, we claim that
    \begin{equation*}
        \reg(S/\calJ_{K_m, G_{v}\setminus v})\le\reg(S/\calJ_{K_m,
        G_{v}})<\reg(S/\calJ_{K_m,G\setminus v}). 
    \end{equation*}
    To confirm this claim, first note  that $G_{v}\setminus v$ is an
    induced subgraph of $G_v$. So we have $\reg(S/\calJ_{K_m,
    G_{v}\setminus v})\le\reg(S/\calJ_{K_m, G_{v}})$ by
    \Cref{cor:induced_graph},  which proves  the first part of the
    claim. Furthermore, $G_{v}$ is the graph $(G_1 \circ \cdots \circ
    G_{t-2}) \circ G'$, where $G'\coloneqq (G_{t-1}\circ G_t)_v$ is a
    $(k_{t-1}+ k_t-2)$-pure fan graph. Under the assumption that  $m\le \min
    \{r_1,r_2,\dots,r_t\}$, one also has $m\le \min \{r_1,\dots,
    r_{t-2},r'\}$, where $r'\coloneqq |V(G')| -(k_{t-1}+ k_t-2)=
    (|V(G_{t-1})|+|V(G_t)|-3)-(k_{t-1}+k_t-2) =r_{t} + r_{t-1}+3$.
    The conditions of \ref{thm:multiple_circ_b_a},
    \ref{thm:multiple_circ_b_b}, and \ref{thm:multiple_circ_b_c} can be checked in a similar way.
    By induction, we now have
    \begin{align}
        \reg(S/\calJ_{K_m, G_v})&=
        \sum_{i=1}^{t-2}\reg(S/\calJ_{K_m,\widetilde{G_i}})
        +\reg(S/\calJ_{K_m,G'\setminus \{v_{t-1,1},f_{t-1,1}\}}),
        \label{eqn:cor_final_1}
    \end{align}
    where $G'\setminus \{v_{t-1,1},f_{t-1,1}\}
    = (G'_{t-1} \circ G_t)_v$
    for $G'_{t-1}\coloneqq G_{t-1}\setminus \{v_{t-1,1},f_{t-1,1}\}$.

    In view of \eqref{eqn:cor_final_bound_1} and
    \eqref{eqn:cor_final_1}, we are reduced to show that
    \begin{equation}
        \reg(S/\calJ_{K_m,(G'_{t-1} \circ G_t)_v})<
        \reg(S/\calJ_{K_m,\widetilde{G_{t-1}}})+
        \reg(S/\calJ_{K_m,\widetilde{G_t}}).
        \label{eqn:cor_final_bound}
    \end{equation}
    For this purpose, let us focus on the graph $G'_{t-1}\circ G_t$.
    Since $G_{t-1}'\setminus \{v_{t-1,2},f_{t-1,2}\}
    =\widetilde{G_{t-1}}$, it is not difficult to see that
    $G'_{t-1}$ is also a $k_{t-1}$-pure pseudo fan graph
    by our assumption. Furthermore, one has
    $m\le\min\{r'_{t-1},r_t\}$ where $r'_{t-1}\coloneqq
    |V(G'_{t-1})|-k_{t-1}=|V(G_{t-1})|-2-k_{t-1}=r_{t-1}+2$. Since
    the conditions of \ref{thm:multiple_circ_b_a},
    \ref{thm:multiple_circ_b_b}, and \ref{thm:multiple_circ_b_c}
    can  be checked in a similar way, it follows from
    \Cref{rem:reg_fan_fan} that \eqref{eqn:cor_final_bound} holds.

    Since we have \eqref{eqn:cor_final_bound}, it follows from
    \Cref{depthlemma}\ref{depthlemma-c} that
    $\reg(S/\calJ_{K_m,G})=\reg(S/\calJ_{K_m,G\setminus v})$,
    which completes the proof. \qedhere
\end{enumerate}
\end{proof}

\begin{Remark}
    Without imposing the condition \ref{thm:multiple_circ_b_d} on $m$ as in
    \Cref{thm:multiple_fan_circ} \ref{thm:multiple_circ_b}, we can obtain strict
    inequalities in \Cref{eqn:multiple_fan_circ-2}. For instance, if $t=2$,
    $G_1=G_2=F_3$ and $m=5$, then $\reg(S/\calJ_{K_m,
    \widetilde{G_1}})+\reg(S/\calJ_{K_m, \widetilde{G_2}})=6<
    \reg(S/\calJ_{K_m,G})=7<2(m+1)=12$, by \Cref{thm:reg_Fp} and
    \Cref{prop:reg_Fan_Fan_circ}.
\end{Remark}

We illustrate the result in \Cref{thm:full_reg_star_circ_path} with the following example.

\begin{Example}
    \begin{enumerate}[a]
        \item Let $G=F_3^{W_1}(K_5)\circ F_3 * F_2^{W_2}(K_4)\circ P_6\circ F_3$ be the graph shown in \Cref{Fig:1}. It can also be recognized as $G=F_3^{W_1}(K_5)\circ F_3 * F_2^{W_2}(K_4)* P_2 *F_3$.
            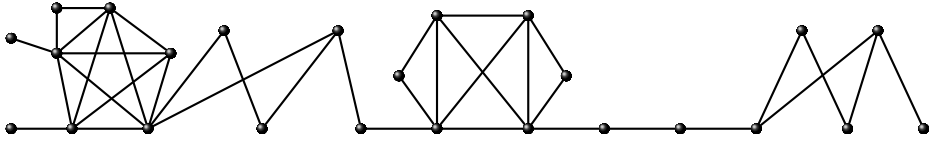
\begin{figure}[htbp]%
                \centering
                \begin{tikzpicture}[thick, scale=1.0, every node/.style={scale=0.98}]]
                    
                    \draw[solid](-5.6,2)--(-4.8,2);
                    \draw[solid](-5.6,3.2)--(-5,3);
                    \draw[solid](-5,3.6)--(-5,3);
                    \draw[solid](-5,3.6)--(-4.3,3.6);
                    \draw[solid](-4.3,3.6)--(-5,3);
                    \draw[solid](-4.3,3.6)--(-4.8,2);
                    \draw[solid](-4.3,3.6)--(-3.8,2);
                    \draw[solid](-5,3)--(-3.5,3);
                    \draw[solid](-5,3)--(-4.8,2);
                    \draw[solid](-5,3)--(-3.8,2);
                    \draw[solid](-4.8,2)--(-3.8,2);
                    \draw[solid](-3.5,3)--(-4.3,3.6);
                    \draw[solid](-3.5,3)--(-4.8,2);
                    \draw[solid](-3.5,3)--(-3.8,2);
                    \draw[solid](-3.8,2)--(-2.8,3.3);
                    \draw[solid](-3.8,2)--(-1.3,3.3);
                    \draw[solid](-2.8,3.3)--(-2.3,2);
                    \draw[solid](-2.3,2)--(-1.3,3.3);
                    \draw[solid](-1.3,3.3)--(-1,2);
                    \draw[solid](0,2)--(-1,2);
                    \draw[solid](0,2)--(0,3.5);
                    \draw[solid](0,2)--(1.2,2);
                    \draw[solid](0,2)--(1.2,3.5);
                    \draw[solid](0,2)--(-0.5,2.7);
                    \draw[solid](0,3.5)--(1.2,2);
                    \draw[solid](0,3.5)--(1.2,3.5);
                    \draw[solid](0,3.5)--(-0.5,2.7);
                    \draw[solid](1.2,3.5)--(1.2,2);
                    \draw[solid](1.7,2.7)--(1.2,2);
                    \draw[solid](1.7,2.7)--(1.2,3.5);
                    \draw[solid](1.2,2)--(2.2,2);
                    \draw[solid](2.2,2)--(3.2,2);
                    \draw[solid](3.2,2)--(4.2,2);
                    \draw[solid](4.2,2)--(4.8,3.3);
                    \draw[solid](4.2,2)--(5.8,3.3);
                    \draw[solid](5.4,2)--(5.8,3.3);
                    \draw[solid](6.4,2)--(5.8,3.3);
                    \draw[solid](4.8,3.3)--(5.4,2);

                    \shade [shading=ball, ball color=black] (-0.5,2.7)  circle (.07);
                    \shade [shading=ball, ball color=black] (-1,2)  circle (.07);
                    \shade [shading=ball, ball color=black] (-1.3,3.3)  circle (.07);
                    \shade [shading=ball, ball color=black] (-2.3,2)  circle (.07);
                    \shade [shading=ball, ball color=black] (-2.8,3.3)  circle (.07);
                    \shade [shading=ball, ball color=black] (-3.5,3)  circle (.07);
                    \shade [shading=ball, ball color=black] (-3.8,2)  circle (.07);
                    \shade [shading=ball, ball color=black] (-4.3,3.6)  circle (.07);
                    \shade [shading=ball, ball color=black] (-4.8,2)  circle (.07);
                    \shade [shading=ball, ball color=black] (-5,3)  circle (.07);
                    \shade [shading=ball, ball color=black] (-5,3.6)  circle (.07);
                    \shade [shading=ball, ball color=black] (-5.6,2)  circle (.07);
                    \shade [shading=ball, ball color=black] (-5.6,3.2)  circle (.07);
                    \shade [shading=ball, ball color=black] (0,2)  circle (.07);
                    \shade [shading=ball, ball color=black] (0,3.5)  circle (.07);
                    \shade [shading=ball, ball color=black] (1.2,2)  circle (.07);
                    \shade [shading=ball, ball color=black] (1.2,3.5)  circle (.07);
                    \shade [shading=ball, ball color=black] (1.7,2.7)  circle (.07);
                    \shade [shading=ball, ball color=black] (2.2,2)  circle (.07);
                    \shade [shading=ball, ball color=black] (3.2,2)  circle (.07);
                    \shade [shading=ball, ball color=black] (4.2,2)  circle (.07);
                    \shade [shading=ball, ball color=black] (4.8,3.3)  circle (.07);
                    \shade [shading=ball, ball color=black] (5.4,2)  circle (.07);
                    \shade [shading=ball, ball color=black] (5.8,3.3)  circle (.07);
                    \shade [shading=ball, ball color=black] (6.4,2)  circle (.07);
                \end{tikzpicture}
                \caption{$F_3^{W_1}(K_5)\circ F_3 * F_2^{W_2}(K_4)\circ P_6\circ F_3$}
                \label{Fig:1}
            \end{figure}
            Using the second viewpoint, \Cref{thm:full_reg_star_circ_path}
            above shows that $\reg(S/\calJ_{K_m,G})\le
            (m+3-1)+(m+2-1)+(m+2-1)+(2-1)+(m+3-1)=4m+7$.

        \item Let $G=F_3*F_4\circ P_7 * F_3\circ F_4$ be the graph shown
            in \Cref{Fig:2}. It is also identified by $F_3*F_4* P_5 *
            F_3\circ F_4$ or $F_3*F_4*F_1*F_1*F_1*F_1*(F_3\circ F_4)$.
            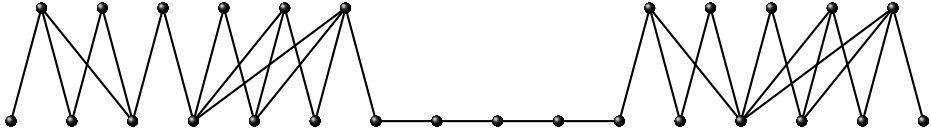
\begin{figure}[htbp]%
                \centering
                \begin{tikzpicture}[thick, scale=1.0, every node/.style={scale=0.98}]]

                    \draw[solid](-6,2)--(-5.6,3.5);
                    \draw[solid](-5.2,2)--(-5.6,3.5);
                    \draw[solid](-5.2,2)--(-4.8,3.5);
                    \draw[solid](-5.6,3.5)--(-4.4,2);
                    \draw[solid](-4.8,3.5)--(-4.4,2);
                    \draw[solid](-4.4,2)--(-4,3.5);
                    \draw[solid](-4,3.5)--(-3.6,2);
                    \draw[solid](-3.6,2)--(-3.2,3.5);
                    \draw[solid](-3.6,2)--(-2.4,3.5);
                    \draw[solid](-3.6,2)--(-1.6,3.5);
                    \draw[solid](-2.8,2)--(-3.2,3.5);
                    \draw[solid](-2.8,2)--(-2.4,3.5);
                    \draw[solid](-2.8,2)--(-1.6,3.5);
                    \draw[solid](-2,2)--(-2.4,3.5);
                    \draw[solid](-2,2)--(-1.6,3.5);
                    \draw[solid](-1.2,2)--(-1.6,3.5);
                    \draw[solid](-1.2,2)--(-0.4,2);
                    \draw[solid](-0.4,2)--(0.4,2);
                    \draw[solid](1.2,2)--(0.4,2);
                    \draw[solid](1.2,2)--(2,2);
                    \draw[solid](2.4,3.5)--(2,2);
                    \draw[solid](2.4,3.5)--(2.8,2);
                    \draw[solid](2.4,3.5)--(3.6,2);
                    \draw[solid](3.2,3.5)--(2.8,2);
                    \draw[solid](3.2,3.5)--(3.6,2);
                    \draw[solid](3.6,2)--(4,3.5);
                    \draw[solid](3.6,2)--(4.8,3.5);
                    \draw[solid](3.6,2)--(5.6,3.5);
                    \draw[solid](4.4,2)--(4,3.5);
                    \draw[solid](4.4,2)--(4.8,3.5);
                    \draw[solid](4.4,2)--(5.6,3.5);
                    \draw[solid](5.2,2)--(4.8,3.5);
                    \draw[solid](5.2,2)--(5.6,3.5);
                    \draw[solid](6,2)--(5.6,3.5);
                    \shade [shading=ball, ball color=black] (-0.4,2) circle (.07);
                    \shade [shading=ball, ball color=black] (-0.4,2)  circle (.07);
                    \shade [shading=ball, ball color=black] (-1.2,2) circle (.07);
                    \shade [shading=ball, ball color=black] (-1.6,3.5)  circle (.07);
                    \shade [shading=ball, ball color=black] (-2,2) circle (.07);
                    \shade [shading=ball, ball color=black] (-2.4,3.5)  circle (.07);
                    \shade [shading=ball, ball color=black] (-2.8,2) circle (.07);
                    \shade [shading=ball, ball color=black] (-3.2,3.5)  circle (.07);
                    \shade [shading=ball, ball color=black] (-3.6,2) circle (.07);
                    \shade [shading=ball, ball color=black] (-3.6,2)  circle (.07);
                    \shade [shading=ball, ball color=black] (-4,3.5) circle (.07);
                    \shade [shading=ball, ball color=black] (-4,3.5)  circle (.07);
                    \shade [shading=ball, ball color=black] (-4.4,2) circle (.07);
                    \shade [shading=ball, ball color=black] (-4.4,2)  circle (.07);
                    \shade [shading=ball, ball color=black] (-4.8,3.5) circle (.07);
                    \shade [shading=ball, ball color=black] (-4.8,3.5)  circle (.07);
                    \shade [shading=ball, ball color=black] (-5.2,2) circle (.07);
                    \shade [shading=ball, ball color=black] (-5.6,3.5) circle (.07);
                    \shade [shading=ball, ball color=black] (-5.6,3.5)  circle (.07);
                    \shade [shading=ball, ball color=black] (-6,2) circle (.07);
                    \shade [shading=ball, ball color=black] (0.4,2)  circle (.07);
                    \shade [shading=ball, ball color=black] (1.2,2) circle (.07);
                    \shade [shading=ball, ball color=black] (2,2)  circle (.07);
                    \shade [shading=ball, ball color=black] (2.4,3.5) circle (.07);
                    \shade [shading=ball, ball color=black] (2.8,2)  circle (.07);
                    \shade [shading=ball, ball color=black] (3.2,3.5) circle (.07);
                    \shade [shading=ball, ball color=black] (3.6,2) circle (.07);
                    \shade [shading=ball, ball color=black] (3.6,2)  circle (.07);
                    \shade [shading=ball, ball color=black] (4,3.5)  circle (.07);
                    \shade [shading=ball, ball color=black] (4.4,2) circle (.07);
                    \shade [shading=ball, ball color=black] (4.8,3.5)  circle (.07);
                    \shade [shading=ball, ball color=black] (5.2,2) circle (.07);
                    \shade [shading=ball, ball color=black] (5.6,3.5) circle (.07);
                    \shade [shading=ball, ball color=black] (5.6,3.5)  circle (.07);
                    \shade [shading=ball, ball color=black] (6,2) circle (.07);
                \end{tikzpicture}
                \caption{$F_3*F_4\circ P_7 * F_3\circ F_4$}
                \label{Fig:2}
            \end{figure}
            Using the second viewpoint, \Cref{thm:full_reg_star_circ_path}
            above shows that $\reg(S/\calJ_{K_m,G})\le
            (m+2-1)+(m+2-1)+(5-1)+(m+2-1)+(m+2-1)=4m+8$.  Note that when
            $m=2$, \cite[Propositions 4.1 and 4.3]{MR3991052} and
            \cite[Theorem 3.1]{MR3941158} show that
            $\reg(S/\calJ_{K_2,G})=\reg(S/J_G)=3+3+1+1+1+1+(3+3)=16$ by using
            the third viewpoint. This value $16$ is exactly $4m+8$ when $m=2$.
    \end{enumerate}
\end{Example}

\section{Dimension}
\label{sec:dim}
In this section, we are interested in the dimension of the generalized binomial edge ideals of graphs in $\mathbf{FFan}$. Unlike the depth or the Castelnuovo--Mumford regularity, the dimension does not behave very well on the short exact sequence of modules. Therefore, we can only deal with a very special case in \Cref{dimension}.  Our result depends on the following simple observations.

\begin{Lemma}
    \label{dim_lem}
    Let $0\rightarrow M \rightarrow N \rightarrow P \rightarrow 0$ be a short exact sequence of finitely generated graded $S$-modules. If $\dim(N)>\dim(P)$, then $\dim(M)=\dim(N)$.
\end{Lemma}
\begin{proof}
    This follows easily from the fact that $\dim(N)=\max\{\dim(M),\dim(P)\}$.
\end{proof}

We start with a simple case.

\begin{Lemma}
    \label{lem:dim_Dv}
    Let $G=F_k^W(K_n)$ be a $k$-fan graph. Suppose that $v\in K_n\setminus W$. Then, $\dim(S_G/\calJ_{K_m,G})> \dim(S_{G\setminus v}/\calJ_{K_m,G\setminus v})$.
\end{Lemma}
\begin{proof}
    Notice that $|V(G\setminus {v})|<|V(G)|$ and $\min\{|W|,n-1\}\le \min\{|W|,(n-1)-1\}$. Thus, the desired inequality follows from \Cref{thm:fan_Dimension}.
\end{proof}

\begin{Lemma} 
    \label{dim}
    Suppose that $G=G_1\circ_v G_2$ where $G_1=F_{p_1}$ with $p_1\ge 3$.
    \begin{enumerate}[a]
        \item \label{dim_1} If  $G_2=F_{p_2}$ with $p_2\ge 2$, then $\dim(S_G/\calJ_{K_m,G})=\sum_{i=1}^2\dim(S_{G_i}/\calJ_{K_m,G_i})-(m+2)$.
        \item \label{dim_2} Suppose that $G_2=F_k^W(K_n)$ is a $k$-fan graph with respect to the decomposition $W=W_1\sqcup \cdots \sqcup W_r$. If $|W|<n$, or $|W|=n$, $|W_1|\ge 2$ and $v\in W_1$, then
            \[
                \dim(S_G/\calJ_{K_m,G})=\sum_{i=1}^2\dim(S_{G_i}/\calJ_{K_m,G_i})-2m.
            \]
    \end{enumerate} 
\end{Lemma}
\begin{proof}  
    Let $H$ be the complete graph on the vertex set $N_G[v]$.  In \Cref{obs:circ}, we have seen that ${G}_v=F_{q}^{W'}(H)$ and ${G}_v\setminus v=F_{q}^{W'}(H\setminus v)$ are two fan graphs, where $q=2$ if $G_2=F_{p_{2}}$, and $q=k$ if $G_2=F_k^W(K_n)$. 
\begin{enumerate}[a]
        \item If $G_2=F_{p_2}$, then $W'=N_G(v)$. In this case, we have $|H|=p_1+p_2-1$ and $|W'|=p_1+p_2-2$. Thus, by  \Cref{thm:fan_Dimension} and \Cref{thm:fp_dim_depth}\ref{thm:fp_dim_depth_a} we have 
            \begin{align*}
                \dim(S_{G_{v}}/\calJ_{K_m,G_{v}}) 
                &= \sum_{i=1}^2\dim(S_{G_i}/\calJ_{K_m,G_i})-(m+2),
            \end{align*}
            since $|V(G_v)|=|V(G_1)|+|V(G_2)|-3$. Likewise, since $G\setminus v=F_{p_1-1}\sqcup F_{p_2-1}$, we can verify that
            \begin{align*}
                \dim(S_{G\setminus v}/\calJ_{K_m,G\setminus v}) & =\sum_{i=1}^2\dim(S_{G_i}/\calJ_{K_m,G_i})-2m. 
            \end{align*}
            Meanwhile, it follows from \Cref{lem:dim_Dv} that 
            \[
                \dim(S_{G_{v}\setminus v}/\calJ_{K_m,G_{v}\setminus v}) <\dim(S_{G_{v}}/\calJ_{K_m,G_{v}}). 
            \]
            In conclusion, we have
            \begin{align*} 
                \dim(S_{G_{v}}/\calJ_{K_m,G_{v}})&=\max \{\dim(S_{G_{v}}/\calJ_{K_m,G_{v}}),\dim(S_{G\setminus v_1}/\calJ_{K_m,G\setminus v})\}\\
                &>\dim(S_{G_{v}\setminus v}/\calJ_{K_m,G_{v}\setminus v}),
            \end{align*} 
            since $m\ge 2$. Thus, the desired result follows from applying \Cref{dim_lem} to the short exact sequence \eqref{eqn:SES-1} in \Cref{rem:exactseq}.

        \item Suppose that $G_2=F_k^W(K_n)$, then $W'=N_{G_1\setminus f_1}(v) \sqcup (W\setminus W_{1})$, where  $f_1$ is a leaf in $G_1$ whose neighbor point is $v$. So $|W'|=(p_1-1)+|W|-|W_1|$.  Furthermore, let $\{K_{a_{1,1}},\ldots,K_{a_{1,|W_1|}}\}$ be the branch of fan in $G_2$ on $W_1$ and write $h_{1,j}\coloneqq  a_{1,j}-j\ge 1$. Then, $|H|=(p_1-1)+n+\sum\limits_{j=1}^{|W_1|}h_{1j}-1$.
            By \Cref{thm:fan_Dimension} and \Cref{thm:fp_dim_depth}\ref{thm:fp_dim_depth_a}, we get that
            \[\
                \dim(S_{G_v}/\calJ_{K_m,G_v})=\sum_{i=1}^2\dim(S_{G_i}/\calJ_{K_m,G_i})-p
            \]
            where 
            \[
                p=\begin{cases}
                    (m+2)+|W_1|(m-2), & \text{if  $|W|<n$,} \\
                    (m+2)+(|W_1|-1)(m-2), & \text{if $|W|=n$ with $|W_1|\ge 2$.}
                \end{cases}
            \]
            It is clear that $p\ge 2m$. At the same time, since $G\setminus v=F_{p_1-1}\sqcup F_{q}^{W\setminus \{v\}}(K_{n-1})$ for suitable $q$, we have 
            \[
                \dim(S_{G\setminus v}/\calJ_{K_m,G\setminus v})=\sum_{i=1}^2\dim(S_{G_i}/\calJ_{K_m,G_i})-2m.
            \]
            Meanwhile, it follows from \Cref{lem:dim_Dv} that 
            \[
                \dim(S_{G_{v}\setminus v}/\calJ_{K_m,G_{v}\setminus v}) <\dim(S_{G_{v}}/\calJ_{K_m,G_{v}}). 
            \]
            In conclusion, we have
            \begin{align*}
                \dim(S_{G\setminus v}/\calJ_{K_m,G\setminus v}) &= \max\{\dim(S_{G_v}/\calJ_{K_m,G_v}),\dim(S_{G\setminus v}/\calJ_{K_m,G\setminus v})\}\\
                & >\dim(S_{G_v\setminus v}/\calJ_{K_m,G_v\setminus v}).
            \end{align*}
            Thus, the expected result follows from applying \Cref{dim_lem} to the short exact sequence \eqref{eqn:SES-1} in \Cref{rem:exactseq}. 
            \qedhere
    \end{enumerate} 
\end{proof}

We can generalize the result in \Cref{dim}.

\begin{Theorem}
    \label{dimension}
    Let $G=G_1\circ_{v_1} \cdots\circ_{v_{t-1}} G_t \circ_{v_t} G_{t+1}$, where  $t\ge 2$ and  $G_i=F_{p_i}$ with $p_i\ge 3$ for $1\le i\le t$. Suppose further that 
    \begin{enumerate}[a]
        \item  $G_{t+1}=F_{p_{t+1}}$ with $p_{t+1}\ge 2$, or
        \item  $G_{t+1}=F_k^W(K_n)$ is a $k$-fan graph with respect to the decomposition $W=W_1\sqcup \cdots \sqcup W_r$, such that either $|W|<n$, or $|W|=n$, $|W_1|\ge 2$, and $v\in W_1$.
    \end{enumerate}
    Then, $\dim(S_G/\calJ_{K_m,G})=\sum_{i=1}^{t+1}\dim(S_{G_i}/\calJ_{K_m,G_i})-s$, where
    \[
        s=
        \begin{cases}
            2m\floor{\frac{t}{2}}+(m+2)\ceil{\frac{t}{2}}, & \text{if $G_{t+1}=F_{p_{t+1}}$,} \\
            2m\ceil{\frac{t}{2}}+(m+2)\floor{\frac{t}{2}}, & \text{if $G_{t+1}=F_k^W(K_n)$.}
        \end{cases}
    \]
As usual, $\floor{\frac{t}{2}}$ here is the largest integer $\le \frac{t}{2}$ and $\ceil{\frac{t}{2}}$ is the smallest integer $\ge \frac{t}{2}$.
\end{Theorem}

\begin{proof}  
    Let $\widetilde{G}=G_{t}\circ_{v_t} G_{t+1}$ and $H$ be the complete graph on the set $N_{\widetilde{G}}[v_t]$.  Then, by  \Cref{obs:circ} we see that  $\widetilde{G}_{v_t}=F_{q}^{W'}(H)$ and  $\widetilde{G}_{v_t}\setminus v_t=F_{q}^{W'}(H\setminus v_t)$ are two fan graphs, where $q=2$ if $G_{t+1}=F_{p_{t+1}}$, and $q=k$ if $G_{t+1}=F_k^W(K_n)$. Furthermore, since $G_{t}=F_{p_t}$ with $p_t\ge 3$, these two fans are with respect to some partition $W'=W_1'\sqcup \cdots \sqcup W_{r'}'$ such that $|W_1'|=p_t-1\ge 2$ and $v_{t-1}\in W_1'$.  
    We have seen in the proof of \Cref{dim} that
    \begin{equation}
        \dim(S_{\widetilde{G}_{v_t}\setminus v_t}/\calJ_{K_m,\widetilde{G}_{v_t}\setminus v_t})< \dim(S_{\widetilde{G}_{v_t}}/\calJ_{K_m,\widetilde{G}_{v_t}}).
        \label{eqn:GDv_smaller}
    \end{equation} 
    We prove the claimed formula by induction $t$. The base case $t=1$ has been verified in \Cref{dim}. Next, we assume that $t\ge 2$. 
    \begin{enumerate}[a]
        \item Suppose that $G_{t+1}=F_{p_{t+1}}$ with $p_{t+1}\ge 2$.  It follows from \Cref{dim}\ref{dim_1} and its proof that $|W'|<|V(H)|$ and
            \begin{equation}
                \dim(S_{\widetilde{G}_{v_t}}/\calJ_{K_m,\widetilde{G}_{v_t}})= \dim(S_{G_t}/\calJ_{K_m,G_t})+\dim(S_{G_{t+1}}/\calJ_{K_m,G_{t+1}})-(m+2).
                \label{eqn:thm_dimension_1}
            \end{equation}
            Since $G_{v_t}=G_1\circ_{v_1}\cdots \circ_{v_{t-2}} G_{{t-1}} \circ_{v_{t-1}} \widetilde{G}_{v_t}$, by induction hypothesis and \Cref{eqn:thm_dimension_1}, we get that 
            \begin{align*} 
                \dim(S_{G_{v_t}}/\calJ_{K_m,G_{v_t}})&=\sum_{i=1}^{t-1}\dim(S_{G_i}/\calJ_{K_m,G_i})+\dim(S_{\widetilde{G}_{v_t}}/\calJ_{K_m,\widetilde{G}_{v_t}})-s\\
                &=\sum_{i=1}^{t+1}\dim(S_{G_i}/\calJ_{K_m,G_i})-s-(m+2),   
            \end{align*}
            for $s\coloneqq 2m\ceil{\frac{t-1}{2}}+(m+2)\floor{\frac{t-1}{2}}$.
            Since $G_{v_t}\setminus v_t=G_1\circ_{v_1}\cdots \circ_{v_{t-2}} G_{{t-1}} \circ_{v_{t-1}} (\widetilde{G}_{v_t}\setminus v_t)$, it also follows from \Cref{eqn:GDv_smaller} and the induction hypothesis that 
            \begin{align*}
                \dim(S_{G_{v_t}\setminus v_t }/\calJ_{K_m,G_{v_t}\setminus v_t})&=\sum_{i=1}^{t-1}\dim(S_{G_i}/\calJ_{K_m,G_i})+\dim(S_{\widetilde{G}_{v_t}\setminus v_t}/\calJ_{K_m,\widetilde{G}_{v_t}\setminus v_t})-s\\
                &<\dim(S_{G_{v_t}}/\calJ_{K_m,G_{v_t}}).
            \end{align*}
            In the end, $G\setminus v_t= G_1\circ_{v_1}\cdots \circ_{v_{t-2}} G_{{t-1}}\circ_{v_{t-1}} (G_t\setminus \{v_t,f_t\}) \sqcup (G_{t+1}\setminus \{v_t,f_{t+1}\})$, where $f_j$ is the leaf in $G_j$ whose neighbor point is $v_t$ for $j\in \{t,t+1\}$. By induction and \Cref{thm:fp_dim_depth}\ref{thm:fp_dim_depth_a}, we have
            \begin{align*}
                \dim(S_{G\setminus v_t}/\calJ_{K_m,G\setminus v_t})&=\sum_{i=1}^{t-1}\dim(S_{G_i}/\calJ_{K_m,G_i})+\dim(S_{G_t\setminus \{v_t,f_t\}/\calJ_{K_m,G_t\setminus \{v_t,f_t\}}})-s'\\
                &\qquad +\dim(S_{G_{t+1}\setminus\{v_t,f_{t+1}\}}/\calJ_{K_m,G_{t+1}\setminus\{v_t,f_{t+1}\}}) \\
                &=\sum_{i=1}^{t+1}\dim(S_{G_i}/\calJ_{K_m,G_i})-s'-2m,
            \end{align*}
            where $s'\coloneqq 2m\floor{\frac{t-1}{2}}+(m+2)\ceil{\frac{t-1}{2}}$. Since $m\ge 2$, we have $s+m+2\le s'+2m$. 
            In conclusion, we have
            \begin{align*}
                \dim(S_{G_{v_t}}/\calJ_{K_m,G_{v_t}}) & =\max\{\dim(S_{G\setminus v_t}/\calJ_{K_m,G\setminus v_t}), \dim(S_{G_{v_t}}/\calJ_{K_m,G_{v_t}})\}\\
                & > \dim(S_{G_{v_t}\setminus v_t }/\calJ_{K_m,G_{v_t}\setminus v_t}).
            \end{align*}
            Therefore, after applying \Cref{dim_lem} to the short exact sequence \eqref{eqn:SES-1} in \Cref{rem:exactseq}, we obtain that
            \begin{align*}
                \dim(S_{G}/\calJ_{K_m,G})&=\dim(S_{G_{v_t}}/\calJ_{K_m,G_{v_t}})=\sum_{i=1}^{t+1}\dim(S_{G_i}/\calJ_{K_m,G_i})-s-(m+2) \\
                &=\sum_{i=1}^{t+1}\dim(S_{G_i}/\calJ_{K_m,G_i})-2m\floor{\frac{t}{2}}-(m+2)\ceil{\frac{t}{2}}.
            \end{align*}
        
        \item Suppose that $G_{t+1}=F_k^W(K_n)$. It follows from the proof of \Cref{dim}\ref{dim_2} that 
            \begin{equation} 
                \dim(S_{\widetilde{G}_{v_t}}/\calJ_{K_m,\widetilde{G}_{v_t}})= \sum_{j=t}^{t+1}\dim(S_{G_j}/\calJ_{K_m,G_j})-p,
                \label{eqn:thm_dimension_2}
            \end{equation}
            for some $p\ge 2m$.
            Since $G_{v_t}=G_1\circ_{v_1}\cdots \circ_{v_{t-2}} G_{{t-1}} \circ_{v_{t-1}} \widetilde{G}_{v_t}$, by induction hypothesis and \Cref{eqn:thm_dimension_2}, we get that 
            \begin{align*} 
                \dim(S_{G_{v_t}}/\calJ_{K_m,G_{v_t}})&=\sum_{i=1}^{t-1}\dim(S_{G_i}/\calJ_{K_m,G_i})+\dim(S_{\widetilde{G}_{v_t}}/\calJ_{K_m,\widetilde{G}_{v_t}})-s\\
                &\le \sum_{i=1}^{t+1}\dim(S_{G_i}/\calJ_{K_m,G_i})-s-2m,   
            \end{align*}
            where $s\coloneqq 2m\ceil{\frac{t-1}{2}}+(m+2)\floor{\frac{t-1}{2}}$.
            Since $G_{v_t}\setminus v_t=G_1\circ_{v_1}\cdots \circ_{v_{t-2}} G_{{t-1}} \circ_{v_{t-1}} (\widetilde{G}_{v_t}\setminus v_t)$, it also follows from \Cref{eqn:GDv_smaller} and the induction hypothesis that 
            \begin{align*}
                \dim(S_{G_{v_t}\setminus v_t }/\calJ_{K_m,G_{v_t}\setminus v_t})&=\sum_{i=1}^{t-1}\dim(S_{G_i}/\calJ_{K_m,G_i})+\dim(S_{\widetilde{G}\setminus v_t}/\calJ_{K_m,\widetilde{G}\setminus v_t})-s \\
                &< \dim(S_{G_{v_t}}/\calJ_{K_m,G_{v_t}}).
            \end{align*}
            In the end, $G\setminus v_t=G_1\circ_{v_1}\cdots \circ_{v_{t-2}} G_{{t-1}}\circ_{v_{t-1}}  (G_t\setminus \{v_t,f_t\}) \sqcup (G_{t+1}\setminus \{v_t,f_{t+1}\})$, where $f_j$ is the leaf in $G_j$ whose neighbor point is $v_t$ for $j\in \{t,t+1\}$. By the induction and \Cref{thm:fp_dim_depth}\ref{thm:fp_dim_depth_a}, we have 
            \begin{align*}
                \dim(S_{G\setminus v_t}/\calJ_{K_m,G\setminus v_t})&=\sum_{i=1}^{t-1}\dim(S_{G_i}/\calJ_{K_m,G_i})+\dim(S_{G_t\setminus \{v_t,f_t\}/\calJ_{K_m,G_t\setminus \{v_t,f_t\}}})-s'\\
                &\qquad +\dim(S_{G_{t+1}\setminus\{v_t,f_{t+1}\}}/\calJ_{K_m,G_{t+1}\setminus\{v_t,f_{t+1}\}}) \\
                &=\sum_{i=1}^{t+1}\dim(S_{G_i}/\calJ_{K_m,G_i}) -s'-2m,
            \end{align*}
where $s'\coloneqq 2m\floor{\frac{t-1}{2}}+(m+2)\ceil{\frac{t-1}{2}}$. Since $s'\le s$, it is clear that 
            \begin{align*}
                \dim(S_{G\setminus {v_t}}/\calJ_{K_m,G\setminus {v_t}}) & =\max \{ \dim(S_{G\setminus v_t}/\calJ_{K_m,G\setminus v_t}), \dim(S_{G_{v_t}}/\calJ_{K_m,G_{v_t}})\}\\
                & >  \dim(S_{G_{v_t}\setminus v_t }/\calJ_{K_m,G_{v_t}\setminus v_t}).
            \end{align*}
            Therefore, it follows from the short exact sequence \eqref{eqn:SES-1} in \Cref{rem:exactseq} that
            \begin{align*}
                \dim(S_{G}/\calJ_{K_m,G})&= \dim(S_{G\setminus {v_t}}/\calJ_{K_m,G\setminus {v_t}})=\sum_{i=1}^{t+1}\dim(S_{G_i}/\calJ_{K_m,G_i}) -s'-2m\\
                &= \sum_{i=1}^{t+1}\dim(S_{G_i}/\calJ_{K_m,G_i})-2m\ceil{\frac{t}{2}}-(m+2)\floor{\frac{t}{2}}. \qedhere
            \end{align*}
    \end{enumerate}
\end{proof}

\begin{Remark}
    When $m=2$, the result in \Cref{dimension} agrees with the depth result in \Cref{thm:full_depth_star_circ}. This provides an alternate proof for the Cohen--Macaulayness of the classical binomial edge ideal of these graphs.
\end{Remark}

\begin{Question}
    What can be said about the dimension of the other graphs in $\mathbf{FFan}$?
\end{Question}

\begin{acknowledgment*}
    This work is supported by the Natural Science Foundation of Jiangsu Province (No.~BK20221353). In addition,  the first author is partially supported by the Anhui Initiative in Quantum Information Technologies (No.~AHY150200) and the ``Innovation Program for Quantum Science and Technology'' (2021ZD0302902). And the second author is supported by the Foundation of the Priority Academic Program Development of Jiangsu Higher Education Institutions.  
\end{acknowledgment*}

\bibliography{BEI}

\begin{bibdiv}
\begin{biblist}

\bib{arXiv:2112.15136}{article}{
      author={Amata, Luca},
      author={Crupi, Marilena},
      author={Rinaldo, Giancarlo},
       title={Cohen--{M}acaulay generalized binomial edge ideal},
        date={2021},
      eprint={arXiv:2112.15136},
}

\bib{MR3859970}{article}{
      author={Baskoroputro, Herolistra},
      author={Ene, Viviana},
      author={Ion, Cristian},
       title={Koszul binomial edge ideals of pairs of graphs},
        date={2018},
        ISSN={0021-8693},
     journal={J. Algebra},
      volume={515},
       pages={344\ndash 359},
         url={https://doi.org/10.1016/j.jalgebra.2018.08.029},
      review={\MR{3859970}},
}

\bib{MR3779601}{article}{
      author={Bolognini, Davide},
      author={Macchia, Antonio},
      author={Strazzanti, Francesco},
       title={Binomial edge ideals of bipartite graphs},
        date={2018},
        ISSN={0195-6698},
     journal={European J. Combin.},
      volume={70},
       pages={1\ndash 25},
         url={https://doi.org/10.1016/j.ejc.2017.11.004},
      review={\MR{3779601}},
}

\bib{MR4423525}{article}{
      author={Bolognini, Davide},
      author={Macchia, Antonio},
      author={Strazzanti, Francesco},
       title={Cohen--{M}acaulay binomial edge ideals and accessible graphs},
        date={2022},
        ISSN={0925-9899},
     journal={J. Algebraic Combin.},
      volume={55},
       pages={1139\ndash 1170},
         url={https://doi.org/10.1007/s10801-021-01088-w},
      review={\MR{4423525}},
}

\bib{MR1613627}{book}{
      author={Brodmann, M.~P.},
      author={Sharp, R.~Y.},
       title={Local cohomology: an algebraic introduction with geometric
  applications},
      series={Cambridge Studies in Advanced Mathematics},
   publisher={Cambridge University Press, Cambridge},
        date={1998},
      volume={60},
        ISBN={0-521-37286-0},
         url={https://doi.org/10.1017/CBO9780511629204},
      review={\MR{1613627}},
}

\bib{MR1251956}{book}{
      author={Bruns, Winfried},
      author={Herzog, J\"{u}rgen},
       title={Cohen--{M}acaulay rings},
      series={Cambridge Studies in Advanced Mathematics},
   publisher={Cambridge University Press, Cambridge},
        date={1993},
      volume={39},
        ISBN={0-521-41068-1},
      review={\MR{1251956}},
}

\bib{MR4233116}{article}{
      author={Chaudhry, Faryal},
      author={Irfan, Rida},
       title={On the generalized binomial edge ideals of generalized block
  graphs},
        date={2020},
        ISSN={1582-3067},
     journal={Math. Rep. (Bucur.)},
      volume={22(72)},
       pages={381\ndash 394},
      review={\MR{4233116}},
}

\bib{MR1627343}{incollection}{
      author={Diaconis, Persi},
      author={Eisenbud, David},
      author={Sturmfels, Bernd},
       title={Lattice walks and primary decomposition},
        date={1998},
   booktitle={Mathematical essays in honor of {G}ian-{C}arlo {R}ota
  ({C}ambridge, {MA}, 1996)},
      series={Progr. Math.},
      volume={161},
   publisher={Birkh\"{a}user Boston, Boston, MA},
       pages={173\ndash 193},
      review={\MR{1627343}},
}

\bib{MR3290687}{article}{
      author={Ene, Viviana},
      author={Herzog, J\"{u}rgen},
      author={Hibi, Takayuki},
      author={Qureshi, Ayesha~Asloob},
       title={The binomial edge ideal of a pair of graphs},
        date={2014},
        ISSN={0027-7630},
     journal={Nagoya Math. J.},
      volume={213},
       pages={105\ndash 125},
         url={https://doi.org/10.1215/00277630-2389872},
      review={\MR{3290687}},
}

\bib{MR2669070}{article}{
      author={Herzog, J\"{u}rgen},
      author={Hibi, Takayuki},
      author={Hreinsd\'{o}ttir, Freyja},
      author={Kahle, Thomas},
      author={Rauh, Johannes},
       title={Binomial edge ideals and conditional independence statements},
        date={2010},
        ISSN={0196-8858},
     journal={Adv. in Appl. Math.},
      volume={45},
       pages={317\ndash 333},
         url={https://doi.org/10.1016/j.aam.2010.01.003},
      review={\MR{2669070}},
}

\bib{MR2643966}{article}{
      author={Hoa, Le~Tuan},
      author={Tam, Nguyen~Duc},
       title={On some invariants of a mixed product of ideals},
        date={2010},
        ISSN={0003-889X},
     journal={Arch. Math. (Basel)},
      volume={94},
       pages={327\ndash 337},
         url={https://doi.org/10.1007/s00013-010-0112-6},
      review={\MR{2643966}},
}

\bib{MR266912}{article}{
      author={Hochster, M.},
      author={Eagon, John~A.},
       title={A class of perfect determinantal ideals},
        date={1970},
        ISSN={0002-9904},
     journal={Bull. Amer. Math. Soc.},
      volume={76},
       pages={1026\ndash 1029},
         url={https://doi.org/10.1090/S0002-9904-1970-12543-5},
      review={\MR{266912}},
}

\bib{MR3991052}{article}{
      author={Jayanthan, A.~V.},
      author={Kumar, Arvind},
       title={Regularity of binomial edge ideals of {C}ohen--{M}acaulay
  bipartite graphs},
        date={2019},
        ISSN={0092-7872},
     journal={Comm. Algebra},
      volume={47},
       pages={4797\ndash 4805},
         url={https://doi.org/10.1080/00927872.2019.1596278},
      review={\MR{3991052}},
}

\bib{JKS}{article}{
      author={Jayanthan, A.~V.},
      author={Kumar, Arvind},
      author={Sarkar, Rajib},
       title={Regularity of powers of quadratic sequences with applications to
  binomial ideals},
        date={2020},
        ISSN={0021-8693},
     journal={J. Algebra},
      volume={564},
       pages={98\ndash 118},
         url={https://doi.org/10.1016/j.jalgebra.2020.08.004},
      review={\MR{4137693}},
}

\bib{MR3941158}{article}{
      author={Jayanthan, A.~V.},
      author={Narayanan, N.},
      author={Raghavendra~Rao, B.~V.},
       title={Regularity of binomial edge ideals of certain block graphs},
        date={2019},
        ISSN={0253-4142},
     journal={Proc. Indian Acad. Sci. Math. Sci.},
      volume={129},
       pages={Paper No. 36, 10 pp},
         url={https://doi.org/10.1007/s12044-019-0480-1},
      review={\MR{3941158}},
}

\bib{arXiv:2207.02256}{article}{
      author={Katsabekis, Anargyros},
       title={Arithmetical rank and cohomological dimension of generalized
  binomial edge ideals},
        date={2022},
      eprint={arXiv:2207.02256},
}

\bib{MR4033090}{article}{
      author={Kumar, Arvind},
       title={Regularity bound of generalized binomial edge ideal of graphs},
        date={2020},
        ISSN={0021-8693},
     journal={J. Algebra},
      volume={546},
       pages={357\ndash 369},
         url={https://doi.org/10.1016/j.jalgebra.2019.10.051},
      review={\MR{4033090}},
}

\bib{MR2782571}{article}{
      author={Ohtani, Masahiro},
       title={Graphs and ideals generated by some 2-minors},
        date={2011},
        ISSN={0092-7872},
     journal={Comm. Algebra},
      volume={39},
       pages={905\ndash 917},
         url={https://doi.org/10.1080/00927870903527584},
      review={\MR{2782571}},
}

\bib{MR3011436}{article}{
      author={Rauh, Johannes},
       title={Generalized binomial edge ideals},
        date={2013},
        ISSN={0196-8858},
     journal={Adv. in Appl. Math.},
      volume={50},
       pages={409\ndash 414},
         url={https://doi.org/10.1016/j.aam.2012.08.009},
      review={\MR{3011436}},
}

\bib{MR3040610}{article}{
      author={Saeedi~Madani, Sara},
      author={Kiani, Dariush},
       title={On the binomial edge ideal of a pair of graphs},
        date={2013},
     journal={Electron. J. Combin.},
      volume={20},
       pages={Paper 48, 13},
      review={\MR{3040610}},
}

\end{biblist}
\end{bibdiv}

\end{document}